\documentclass[11pt]{amsart}

\usepackage{amsmath}
\usepackage{amssymb}
\usepackage{amsthm}
\usepackage{mathrsfs}
\usepackage{graphicx}
\usepackage{xcolor}
\usepackage[colorlinks=true,allcolors=black]{hyperref}

\newtheorem{theorem}{Theorem}[section]
\newtheorem{mainthm}{Theorem}
\newtheorem{lemma}[theorem]{Lemma}
\newtheorem{proposition}[theorem]{Proposition}
\newtheorem{corollary}[theorem]{Corollary}
\theoremstyle{definition}
\newtheorem{definition}[theorem]{Definition}
\newtheorem{remark}[theorem]{Remark}

\newcommand{\R}{\mathbb{R}}
\newcommand{\T}{\mathbb{T}}
\newcommand{\Z}{\mathbb{Z}}
\newcommand{\C}{\mathbb{C}}
\newcommand{\Sp}{\mathfrak{S}}
\DeclareMathOperator{\tr}{tr}

\begin{document}

\title{
Global solutions for the Alber equation \\ in $H^1\Sp^1(\mathbb{T})$}
\author{Agissilaos Athanassoulis}
\address{Department of Mathematics, University of Dundee}
\email{aathanassoulis@dundee.ac.uk}
\date{\today}

\subjclass[2020]{Primary 35Q55; Secondary 35B35, 47B10, 37A60}
\keywords{Alber equation, mixed-state NLS, Schatten--Sobolev,
global existence, energy conservation, a priori bound, focusing}

\begin{abstract}
The Alber equation is the mixed-state  nonlinear Schr\"odinger equation with singular ($\delta$-interaction) kernel. It is used in the modeling of stochastic ocean waves, where it appears with the focusing sign in the nonlinearity, on $d=1.$ 
The main result of the paper is global well-posedness for self-adjoint, non-negative data in the Schatten-Sobolev space $H^1\mathfrak{S}^1(\T),$ for both the focusing and defocusing cases.
The Schatten class norms  achieve  control of the position density without derivative loss, and
a systematic Fourier-Galerkin argument tailored to the $\delta$ kernel allows us to establish several qualitative properties of the solution, including  energy conservation. In the focusing case, Hoffmann-Ostenhof and Gagliardo-Nirenberg estimates yield a global a priori $H^1\mathfrak{S}^1$ bound with no smallness condition. Non-negativity is a structural requirement for the energy argument to work. The propagation of higher Sobolev regularity $H^s\mathfrak{S}^1$   follows. As an application, small perturbations around Penrose-stable backgrounds are  shown to grow at most polynomially in $H^1\mathfrak{S}$ over long timescales.
\end{abstract}

\maketitle

\section{Introduction}\label{sec:intro}

\subsection{The equation}
The nonlinear Schr\"odinger equation (NLS)
\begin{equation}\label{eq:nls}
	i \partial_t \psi = p\Delta \psi + q |\psi|^2 \psi
\end{equation}
appears in quantum mechanics in the study of condensates and mean field limits. It also appears in a plethora of other fields (ocean waves, nonlinear optics, Langmuir waves etc) as a canonical nonlinear dispersive model \cite{MeiStiassnieYue2005b,SulemSulem1999}. When $pq>0,$ the equation is focusing, and the nonlinear focusing is competing with the linear dispersion. This leads to more complex energy balance, in contrast to the defocusing case,   $pq<0.$ 

Using von Neumann's distinction of the pure state  (which can be expressed in terms only of a wavefunction) $\psi$ and mixed state (necessarily an operator) $\gamma=\sum\limits\lambda_n|\phi_n\rangle\langle\phi_n|$  \cite{vonNeumann1927}, equation  \eqref{eq:nls} is the pure state problem, and the mixed state version of that problem reads
\begin{equation}\label{eq:Alber_op}
  i\partial_t \gamma = p[\Delta,\gamma]
  + q[V_{\rho_\gamma},\gamma],
\end{equation}
where $\gamma(t)$ is, by construction, a self-adjoint operator. Denoting the integral  kernel of $\gamma(t)$ by $K_\gamma(x,y,t),$ its {\em position density} $\rho_\gamma$ is defined as
\begin{equation}\label{eq:def_posden}
\rho_\gamma(x,t)=K_\gamma(x,x,t),
\end{equation}
and $V_f$ denotes the operator of multiplication by $f$. The mixed state problem \eqref{eq:Alber_op} can equivalently be formulated as a nonlocal PDE for the kernel,
\begin{equation}\label{eq:Alber0}
  i\partial_t K_\gamma = p(\Delta_x - \Delta_y)K_\gamma
  + q\bigl(\rho_\gamma(x,t)-\rho_\gamma(y,t)\bigr)K_\gamma.
\end{equation}
In the wider context of Hartree-type equations this is often called the {\em `singular kernel case'}, in the sense that Hartree equations have similar structure as \eqref{eq:Alber0}, but with $w\ast_x \rho_\gamma(x,t)$ in place of $\rho_\gamma(x,t)$ -- where $w$ is a  smoothing kernel.

The mixed-state formalism describes a statistical ensemble of wavefunctions, equivalent to the infinite coupled system
\begin{equation}\label{eq:Alber_system}
  i\partial_t \phi_n = p\Delta \phi_n + qV_\rho \phi_n,
  \qquad
  \rho(x,t) = \sum_{n\in\Z} \lambda_n |\phi_n(x,t)|^2,
  \qquad n\in\Z,
\end{equation}
where $\{\phi_n(t)\}\subset L^2(\T)$ is an orthonormal family and the weights satisfy $\lambda_n\ge 0,$ $\sum_n\lambda_n<\infty$. The non-negativity of the operator is encoded in the sign of the eigenvalues (and we can drop any zero eigenvalues to have  $\lambda_n> 0$ without loss of generality).  In the formalism of equation~\eqref{eq:Alber_op}, the wavefunctions $\phi_n$ correspond to the eigenfunctions of the operator $\gamma$ via the spectral decomposition $\gamma=\sum_n\lambda_n|\phi_n\rangle\langle\phi_n|$; thus \eqref{eq:Alber_op} encodes the infinite system~\eqref{eq:Alber_system} as a single operator equation.

Equation \eqref{eq:Alber0} admits homogeneous, steady-state  solutions of the form $\Gamma(x-y);$ a central question is the stability of such solutions. This leads to considering initial data of the form $K_\gamma(x,y,0) = \Gamma(x-y) + u(x,y,0),$ where $u$ denotes the {\em inhomogeneous perturbation} and $\Gamma$ the {\em homogeneous background}. The inhomogeneity $u$ therefore satisfies
\begin{equation}\label{eq:Alber1}
  i\partial_t u = p(\Delta_x - \Delta_y)u
  + q\bigl(u(x,x,t)-u(y,y,t)\bigr)\bigl(\Gamma(x-y)+u\bigr).
\end{equation}
Such problems appear in ocean waves \cite{Alber1978}, where the problem is naturally focusing, and in mean field models  \cite{ChenHongPavlovic2017,HadamaHong2025,HadamaRout2026}, where the problem is naturally defocusing. In both cases, the stability of homogeneous solutions is a key physical question.

There are two main formulations for the Alber equation, the operator valued \eqref{eq:Alber_op} and the PDE for the kernel \eqref{eq:Alber0}. Moreover, there is the stability-of-homogeneous-backgrounds version \eqref{eq:Alber1}, and its operator-valued counterpart. These equations appear with different names in different papers. In \cite{ChenHongPavlovic2017}, where only the defocusing case is considered,  it is called the {\em `infinite system of fermions'}. In \cite{HadamaHong2025}, where again only the defocusing case is considered -- investigating a broader class of backgrounds $\Gamma$ -- it is called {\em `the nonlinear Hartree equation with infinitely many particles and singular interaction'}. In an established ocean waves literature, where always the focusing case is considered, it is called the {\em `Alber equation'} after \cite{Alber1978}; cf. for example \cite{AAPS2020,AthanassoulisKyza2024,RibalBabaninYoungToffoliStiassnie2013,StiassnieRegevAgnon2008} and the references therein. In that context, $d=1$ corresponds to unidirectional wave propagation, which is a common scenario. To emphasize the singular interaction, the low dimension $d=1,$ and the inclusion of  the focusing case, we use the term {\em Alber equation} for all of these formulations. 

\subsection{State of the art}
The first question about the Alber equation \eqref{eq:Alber_op}, or equivalently \eqref{eq:Alber0}, is the well-posedness of the Cauchy problem. Local-in-time well-posedness can be shown on full space $(x,y)\in\mathbb{R}^{2d}$ \cite{AAPS2020} or the torus $(x,y)\in\mathbb{T}^{2d}$ \cite{AthanassoulisKarakatsaniKyza2025} by standard methods. 
Numerical evidence \cite{AthanassoulisKarakatsaniKyza2025,RibalBabaninYoungToffoliStiassnie2013,StiassnieRegevAgnon2008} shows no blow-up or ill-posed behavior even for strongly Penrose-unstable backgrounds. 

When trying to prove global existence of solutions, a  natural idea is to try and use the problem's own energy. This idea has a long history for Hartree equations \cite{BoveDaPratoFano1976,Zagatti1992}. The first issue is to rigorously establish the (formally obvious) conservation of energy -- this is where the regularity of the interaction kernel first enters the picture. In the focusing case, another key issue is to control one term of the energy by the other, so as to obtain an a priori bound for the solution. The aforementioned works \cite{BoveDaPratoFano1976,Zagatti1992} deal with  Coulomb and $L^p$ interaction kernels (no $\delta$ kernels).

More recently \cite{ChenHongPavlovic2017,HadamaHong2025,LewinSabin2015,LewinSabin2014},   homogeneous backgrounds  were added to the problem. On $\R^d,$ this necessarily leads  to infinite total mass and energy; this was addressed by introducing a {\em relative energy} for the perturbation in \cite{LewinSabin2015,LewinSabin2014}; however this also meant that the scope had to be limited to  defocusing problems only. The works  \cite{ChenHongPavlovic2017,HadamaHong2025} expand the approach to $\delta-$interaction kernels by careful regularization for the relative energy  on $\R^3$ --  for defocusing problems, and under the Pauli exclusion principle. 

In this paper we work on $\T,$ so even homogeneous backgrounds have finite total mass and  energy. This allows us to work with the {\em total} energy, and thus brings focusing problems within reach. Working on $\T$ also provides analytical underpinning for numerical investigations of the problem \cite{AthanassoulisKarakatsaniKyza2025,RibalBabaninYoungToffoliStiassnie2013,StiassnieRegevAgnon2008}, which necessarily take place on bounded domains. 
\medskip 

We should also mention
other approaches in related settings. One idea is to mirror the Landau-damping approach in classical kinetic equations: go to the form \eqref{eq:Alber1} and examine its linear stability, i.e., whether the solutions of
\begin{equation}\label{eq:Alber1lin}
  i\partial_t u_{\mathrm{lin}} = p(\Delta_x - \Delta_y)u_{\mathrm{lin}}
  + q\bigl(u_{\mathrm{lin}}(x,x,t)-u_{\mathrm{lin}}(y,y,t)\bigr) \Gamma(x-y)
\end{equation}
grow. This yields a Penrose-type linear stability condition for $\Gamma,$ followed by a perturbative nonlinear stability analysis. This is carried out in the preprint \cite{Smith2024}, and it includes focusing problems in $d\ge 3,$ but not $\delta$-interaction kernels. 

Finally, the recent preprint \cite{HadamaRout2026} proves global existence of solutions for both the focusing and defocusing problem with $\delta$ interaction kernel on $\T$ at the $\Sp^1(\T)$ level. Their approach is based on renormalized Strichartz estimates on $\T.$ 
To the best of our knowledge, \cite{HadamaRout2026} and the present paper are the only works containing a global well-posedness result for the focusing mixed-state cubic NLS in $d=1.$

\subsection{Ocean waves context}\label{subsec:ocean}

Alber's original derivation \cite{Alber1978} started from the Davey-Stewartson system, which governs the envelope in quasi-2D, with a main direction of propagation and   transverse dispersion. However, the transverse dependence can be integrated away, reducing the problem to one spatial dimension along the main direction of propagation. Much of the subsequent literature works in the setting $d=1$ \cite{AAPS2020,RibalBabaninYoungToffoliStiassnie2013,StiassnieRegevAgnon2008},  and this is the setting for the present paper. Genuine two-dimensional extensions  include crossing seas (i.e.\ two wavesystems carrying $O(1)$ energy traveling in two genuinely different directions) \cite{AthanassoulisGramstad2021}, as well as fully two-dimensional distributions of energy \cite{AndradeStiassnie2020,CrawfordSaffmanYuen1980}.

The main role of the  Alber equation is to refine the study of the modulation instability  (the linear instability of plane waves) \cite{MeiStiassnieYue2005b,MI,ZakharovOstrovsky2009} to realistic ocean wave spectra -- which can be narrow, but not quite plane waves \cite{Ochi} . This question is precisely the stability of homogeneous backgrounds, and is most interesting around backgrounds that are close to the stability threshold -- on either side of it.
This bifurcation has direct consequences for marine safety, as unstable sea states are believed to promote the formation of rogue waves \cite{AthanassoulisKarakatsaniKyza2025,Dysthe2008,RibalBabaninYoungToffoliStiassnie2013,StiassnieRegevAgnon2008}. Theorem~\ref{thm:nlstab} provides a first rigorous nonlinear stability result in the Penrose-stable regime; we expect that this could be substantially improved in future work. More broadly, this paper puts the numerical and asymptotic treatment of this problem on rigorous footing.

\section{Main results}\label{sec:mainresults}

Here we will go over the main results, together with the strategy of proof and a roadmap of the paper.
Note that  the energy associated with \eqref{eq:Alber_op} is
\begin{equation}\label{eq:E_def}
  E(\gamma) = p\,\tr(\Delta\,\gamma) + \frac{q}{2}\|\rho_\gamma\|_{L^2}^2.
\end{equation}
The energy is finite for every self-adjoint, non-negative $\gamma\in H^1\Sp^1(\T)$ (cf. Lemma~\ref{lem:E_finite}). 
Proving rigorously the conservation of the energy is a key step.

\subsection{Statement of the main results}\label{subsec:main_thms}

\begin{mainthm}[Global well-posedness for the Alber equation, arbitrary data]\label{thm:main}
Let $pq\ne 0$, and let $\gamma_0\in H^1\Sp^1(\T)$ be self-adjoint and non-negative. Then:
\begin{itemize}
\item[(a)] The Alber equation~\eqref{eq:Alber_op} has a unique global mild solution
\[
  \gamma\in C(\R;\,H^1\Sp^1(\T)),
\]
which conserves mass $\|\gamma(t)\|_{\Sp^1}$, Hilbert--Schmidt norm $\|\gamma(t)\|_{\Sp^2}$, and energy $E(\gamma(t));$ and there exists a constant $\bar Y=\bar Y(\|\gamma_0\|_{H^1\Sp^1},p,q,|\T|),$ given explicitly in Remark~\ref{rem:barY} (with distinct closed forms in the focusing case $pq>0$ and the defocusing case $pq<0$), such that
\[
\|\gamma(t)\|_{H^1\Sp^1} \le \bar Y \quad \forall t \in \R.
\]
\item[(b)] For every $s\ge 2,$ if $\gamma_0\in H^s\Sp^1(\T),$ then
\[
  \gamma\in C(\R;\,H^s\Sp^1(\T))\cap C^1(\R;\,H^{s-2}\Sp^1(\T)),
\]
with an at-most exponential-in-$t$ bound on $\|\gamma(t)\|_{H^s\Sp^1}$ (with rate governed by $\bar Y$), and~\eqref{eq:Alber_op} holds in the strong sense in $H^{s-2}\Sp^1.$
\end{itemize}
\end{mainthm}


\smallskip 

The following  is a first example of how the global $H^1\Sp^1$ bound can be used for stable backgrounds:

\begin{mainthm}[Polynomial-in-$t$ control near Penrose-stable backgrounds]
\label{thm:nlstab}
Let $pq\ne 0$ (either sign), let $\Gamma$ be a Penrose-stable background on $\T$ with constant $\kappa>0$ in the sense of Definition~\ref{def:penrose} (in particular $\Gamma=\Gamma^*\ge 0$ and $\Gamma\in H^1\Sp^1$), and let $\varepsilon>0.$ For every self-adjoint perturbation $U_0\in H^1\Sp^1(\T)$ with $\|U_0\|_{H^1\Sp^1}\le\varepsilon$ and $\gamma_0:=\Gamma+U_0\ge 0,$ the solution $\gamma(t)$ of~\eqref{eq:Alber_op} exists globally in $C(\R;\,H^1\Sp^1(\T))$ and
\begin{equation}\label{eq:nlstab_bound}
  \|\gamma(t)-\Gamma\|_{H^1\Sp^1}
  \;\le\; 2\,C_\star(\Gamma)\,
  \langle t\rangle^2\,\varepsilon
  \qquad\text{for all }|t|\le T_\star(\Gamma,\varepsilon),
\end{equation}
where
\begin{equation}\label{eq:Tstar_def}
  T_\star(\Gamma,\varepsilon)
  \;=\; c_\Gamma\,\varepsilon^{-1/5},
\end{equation}
and $C_\star(\Gamma)=C_\star(\Gamma,p,q,\kappa)$ and $c_\Gamma=c_\Gamma(\Gamma,p,q,\kappa)$ are positive constants independent of $\varepsilon.$ In particular, 
\[
\|\gamma(t)-\Gamma\|_{H^1\Sp^1}\lesssim\varepsilon^{3/5}
\]
uniformly on the time window $|t|\le T_\star$.
\end{mainthm}

\begin{remark}[Theorem \ref{thm:nlstab} is more than $H^1\Sp^1$ well-posedness]
\label{rem:small_not_penrose}

 The well-posedness of the problem, by virtue of Theorem~\ref{thm:main} applied to the setting of Theorem \ref{thm:nlstab}, automatically yields an exponential bound of the form
 \[
\|\gamma(t)-\Gamma\|_{H^1\Sp^1}\le \varepsilon e^{Ct}.
 \]
There exist Penrose-unstable backgrounds $\Gamma\in H^1\Sp^1(\T)$ for which $\|\gamma(t)-\Gamma\|_{H^1\Sp^1}\gtrsim\varepsilon e^{ct}$ — an explicit example is constructed in Remark~\ref{rem:explicit_unstable}. In the ocean-wave context, such exponential growth of sideband perturbations is the classical modulation instability \cite{MI,ZakharovOstrovsky2009}. This shows that the exponential growth predicted by raw well-posedness can in fact be saturated on Penrose-unstable backgrounds, and highlights that Theorem~\ref{thm:nlstab} provides genuinely better control in the Penrose-stable case than the one obtained from well-posedness alone. We believe that Theorem~\ref{thm:nlstab} may be substantially improved, possibly to genuine orbital stability, $\|\gamma(t)-\Gamma\|_{H^1\Sp^1}=o(1)$ as $t\to\infty.$
\end{remark}

\subsection{Strategy of proof}\label{subsec:strategy}

The proof of Theorem~\ref{thm:main} proceeds as follows. We show local well-posedness at the operator level \eqref{eq:Alber_op}. To extend to all time, we need several qualitative properties --- chiefly non-negativity (a prerequisite for many estimates) and conservation of energy. To establish these properties, we apply a Fourier-Galerkin projection to the initial datum $\gamma_0$ and form a finite-rank system of the form \eqref{eq:Alber_system}. Lemma~\ref{lem:finite_NLS} establishes global well-posedness and qualitative properties for any smooth finite-rank system, and Lemma~\ref{lem:finite_system} lifts these to the operator-valued solution in the limit. Thus, the work of showing the energy conservation and the other qualitative properties takes place at the level of a finite-rank system, and is then passed back up to the operator level by $H^1\Sp^1$ continuity. Section~\ref{sec:proof_thms} assembles these results into the global-in-time argument.

\smallskip

We now sketch the argument for the a priori $H^1\Sp^1$ bound, for the focusing case. We will use the conservation of mass $\|\gamma(t)\|_{\Sp^1}=\mathcal{M},$ and of the energy~\eqref{eq:E_def} which in particular implies
\begin{equation}\label{eq:pke1}
  |p|\,\tr(-\Delta\gamma(t))
  \;\le\; |E(\gamma_0)| + \frac{|q|}{2}\,\|\rho_{\gamma(t)}\|_{L^2}^2.
\end{equation}
Thus, closing the bound on the kinetic energy
\[
\mathcal{K}(t):=\tr(-\Delta\gamma(t))
\] 
reduces to controlling the nonlinear potential energy $\|\rho_{\gamma(t)}\|_{L^2}^2$ by  $\mathcal{K}(t)$. The na\"ive estimate $\|\rho_\gamma\|_{L^2}^2\le\|\rho_\gamma\|_{L^1}\|\rho_\gamma\|_{L^\infty}\le C\,\mathcal{M}\bigl(\mathcal{M}+\mathcal{K}(t)\bigr)$ via Lemma \ref{lem:Bessel}, is \emph{linear} in $\mathcal{K}(t)$ and closes only for sufficiently small mass.
Instead, we will  combine the  Hoffmann--Ostenhof inequality~\cite{HoffmannOstenhof1977}
\[
  \int_\T\bigl|\nabla\sqrt{\rho_\gamma}\bigr|^2\,dx
  \;\le\;
  \tr(-\Delta\gamma),
  \qquad \gamma\ge 0,
\]
and the Gagliardo-Nirenberg inequality
\[
  \|u\|_{L^4(\T)}^{\,4}
  \;\le\; \frac{1}{|\T|}\,\|u\|_{L^2}^4 + 2\,\|u\|_{L^2}^3\,\|\nabla u\|_{L^2},
\]
to the effect of\begin{align*}
  \|\rho_\gamma\|_{L^2}^2 \;=\; \|\sqrt{\rho_\gamma}\|_{L^4}^4
  &\;\stackrel{\text{GN}}{\le}\;
  C\,\|\sqrt{\rho_\gamma}\|_{L^2}^3\,\|\nabla\sqrt{\rho_\gamma}\|_{L^2}
  + C\,\frac{\|\sqrt{\rho_\gamma}\|_{L^2}^4}{|\T|}\\
  &\;\stackrel{\text{HO}}{\le}\;
  C\,\mathcal{M}^{3/2}\,\sqrt{\mathcal{K}(t)}
  + C\,\frac{\mathcal{M}^2}{|\T|},
\end{align*}
using $\|\sqrt{\rho_\gamma}\|_{L^2}^2=\|\rho_\gamma\|_{L^1}=\mathcal{M}$ in the last step.
This, combined with \eqref{eq:pke1}, leads to the \emph{sublinear} bound
\[
\mathcal{K}(t) \le A + B \sqrt{\mathcal{K}(t)}
\]
for appropriate $A,B>0,$
which allows a global a priori bound in the energy without smallness assumption.
The argument is carried out in detail in Step~4 of Lemma~\ref{lem:finite_NLS}, on the level of a finite rank system.

In the defocusing case ($pq<0$), the kinetic and potential terms in~\eqref{eq:E_def} carry the same sign, so $|p|\,\mathcal{K}(t)\le|E(\gamma_0)|$ directly from energy conservation, and the a priori bound is direct.

\smallskip 

Theorem~\ref{thm:nlstab} is proved in Section~\ref{subsec:proof_nlstab}. It uses the Schatten-class bounds for the linearized propagator that are established in Section~\ref{subsec:setup_linstab}. This requires a shifted Bromwich-contour bound for the linearized propagator, optimized in the shift $\eta=\eta(t)\sim 1/\langle t\rangle$, producing an algebraic time window on which the perturbation grows at most polynomially. The argument is sign-agnostic and handles both the focusing and defocusing cases uniformly. While this result is short of actual orbital stability, it demonstrates that it is possible in principle to use the Penrose stability in terms of the Schatten-Sobolev norms, and thus directly exploit  the energy bound.
Finally, the appendix collects some standard technical inputs on the Volterra kernel $\Phi_k$ used in the proof of Proposition~\ref{prop:propagator}.


\subsection{Organization of the paper}
Section~\ref{sec:framework} introduces the functional framework (Schatten--Sobolev spaces, and estimates that will be used later). Section~\ref{sec:galerkin} sets up the mild formulation and local well-posedness, then develops the Fourier--Galerkin approximation: the static truncation (Proposition~\ref{prop:galerkin}), the smooth orbital NLS engine (Lemma~\ref{lem:finite_NLS}), and the bridge lemma (Lemma~\ref{lem:finite_system}) which lifts conservation laws, positivity, a priori bounds, and higher-regularity estimates from the smooth approximant to any $H^1\Sp^1$ solution. The proof of Theorem~\ref{thm:main} is carried out in Section~\ref{sec:proof_thms}, using heavily the previous Sections. Section~\ref{sec:stab} proves Theorem~\ref{thm:nlstab} by establishing a propagator estimate for the linearization around a Penrose-stable background and closing a bootstrap argument. Appendix~\ref{app:volterra} collects the properties of the Volterra kernel $\Phi_k$ and the auxiliary lemmas used in the proof of the propagator estimate.

\section{Functional framework and notation}\label{sec:framework}

\subsection{Notation and Schatten--Sobolev spaces}\label{subsec:notation}

Without loss of generality, we identify $\T = \R/(2\pi\Z)$ throughout the paper, and use the unitary Fourier convention
\[
  \hat f(n) = \frac{1}{\sqrt{2\pi}}\int_0^{2\pi} f(x)\,e^{-inx}\,dx,
  \qquad
  f(x) = \frac{1}{\sqrt{2\pi}}\sum_{n\in\Z}\hat f(n)\,e^{inx},
\]
under which $\|f\|_{L^2(\T)}^2 = \sum_n|\hat f(n)|^2$ and $-\Delta$ has eigenvalues $n^2$ on the modes $e^{inx}.$ We are using throughout the Japanese bracket
\[
\langle\xi\rangle=(1+|\xi|^2)^{1/2}
\]
and  the Bessel potential operator
\[
\langle D\rangle=(1-\Delta)^{1/2},
\]
which acts as $\langle n\rangle = (1+n^2)^{1/2}$ on $e^{inx}.$ For $x_0\in\T,$ we write $\delta_{x_0}=\delta(\,\cdot\,-x_0)$ for the Dirac delta centred at $x_0.$

We write $C$ for a positive constant that may change from line to line; named constants such as $B_s$ (the Bessel-kernel constant of Lemma~\ref{lem:Bessel}) retain their values throughout, and $B_1$ denotes its specialization at $s=1.$

\begin{definition}[Multiplication operator $V_f$]\label{def:Vf}
For $f\in L^\infty(\T),$ $V_f$ denotes the multiplication operator $V_f:L^2(\T)\to L^2(\T),$ $h\mapsto fh.$
\end{definition}

\begin{definition}[Operator norm]\label{def:op_norm}
The operator norm on $L^2(\T)$ is denoted $\|\cdot\|_{\mathrm{op}};$ it coincides with the $\Sp^\infty$ norm as defined below.
\end{definition}

\begin{definition}[Schatten spaces]
For $1\le p<\infty$, the Schatten class $\Sp^p=\Sp^p(L^2(\T))$ consists of compact operators $A$ for which $\tr|A|^p<\infty$, with norm $\|A\|_{\Sp^p}=\bigl(\tr|A|^p\bigr)^{1/p}$, where $|A|=(A^*A)^{1/2}.$ In particular $\Sp^1$ is the trace class and $\Sp^2$ is the Hilbert--Schmidt class (isometric to $L^2(\T^2)$ via the kernel). We set $\Sp^\infty:=B(L^2(\T))$, the bounded operators on $L^2(\T)$ with the operator norm. The inclusion $\Sp^1\subset\Sp^2\subset\Sp^\infty$ holds with $\|\cdot\|_{\Sp^\infty} \le\|\cdot\|_{\Sp^2} \le\|\cdot\|_{\Sp^1}$.
\end{definition}

The free evolution (propagator of \eqref{eq:Alber_op} for $q=0$) is the conjugation
\begin{equation}\label{eq:free}
  S(t)\gamma_0 = e^{-ip\Delta t}\,\gamma_0\,e^{ip\Delta t},
\end{equation}
which is an isometry on every Schatten class.

\begin{definition}[Schatten--Sobolev spaces]
For $s\ge 0$ and $1\le r\le\infty$, define the symmetric Schatten--Sobolev space and norm as
\begin{equation}\label{eq:HsSr}
  H^s\Sp^r = \bigl\{\gamma\in\Sp^r :
  \langle D\rangle^s\gamma\langle D\rangle^s
  \in\Sp^r\bigr\},
  \quad
  \|\gamma\|_{H^s\Sp^r}
  = \|\langle D\rangle^s\gamma
  \langle D\rangle^s\|_{\Sp^r}.
\end{equation}
\end{definition}

The free evolution~\eqref{eq:free} commutes with $\langle D\rangle$, hence $S(t)$ is isometric on $H^s\Sp^r$ for every $s$ and $r$.

\begin{remark}[Kernel regularity]\label{rem:kernel_reg}
If $\gamma\in H^s\Sp^1$, then $\langle D_x\rangle^s\langle D_y\rangle^s K_\gamma\in L^2(\T^2)$ (via $\Sp^1\subset\Sp^2\cong L^2(\T^2)$), i.e.\ $K_\gamma\in H^s_x H^s_y(\T^2)$. Since $\langle\xi\rangle^s\langle\eta\rangle^s \ge\langle(\xi,\eta)\rangle^s$ for $s\ge 0,$ 
\[
\gamma \in H^s\Sp^1(\T)\implies K_\gamma \in  H^s(\T^2)
\]
\end{remark}

\subsection{Background material}\label{subsec:background}

We collect here well known results used throughout the paper.

\begin{remark}[Sobolev embedding]\label{rem:sobolev}
Since $d=1$ and $s>1/2$, the Sobolev embedding $H^s(\T)\hookrightarrow L^\infty(\T)$ holds; we use $\|f\|_{L^\infty}\le C\|f\|_{H^s}$ throughout without further comment.
\end{remark}

\begin{proposition}[Schatten--H\"older inequality]\label{prop:schatten_holder}
Let $r,r_1,r_2\in[1,\infty]$ with $\frac{1}{r}=\frac{1}{r_1}+\frac{1}{r_2},$ and let $A\in\Sp^{r_1}$, $B\in\Sp^{r_2}.$ Then $AB\in\Sp^r$ and
\begin{equation}\label{eq:schatten_holder}
  \|AB\|_{\Sp^r}\le\|A\|_{\Sp^{r_1}}\|B\|_{\Sp^{r_2}}.
\end{equation}
\end{proposition}
\noindent{See \cite[Thm.~2.8]{Simon2005}.}

\begin{lemma}[Kato--Ponce fractional Leibniz rule at $(2,\infty)$]\label{lem:KP}
Let $s>0$. For all $f,g\in H^s(\T)\cap L^\infty(\T),$
we have
\begin{equation}\label{eq:KP}
  \|\langle D\rangle^s(fg)\|_{L^2}
  \le C\bigl(
  \|f\|_{H^s}\,\|g\|_{L^\infty}
  + \|f\|_{L^\infty}\,\|g\|_{H^s}
  \bigr)
\end{equation}
and
\begin{equation}\label{eq:KP_commutator}
  \|[\langle D\rangle^s, V_f]\,g\|_{L^2}
  \le C\bigl(
  \|f\|_{H^s}\,\|g\|_{L^\infty}
  + \|f\|_{L^\infty}\,\|g\|_{H^s}
  \bigr).
\end{equation}
\end{lemma}
\noindent{See \cite[Thm.~1]{GrafakosOh2014} for~\eqref{eq:KP}}. Equation \eqref{eq:KP_commutator} is an elementary corollary of (i), by writing
\[
  [\langle D\rangle^s,V_f]g\;=\;\langle D\rangle^s(fg)\;-\;f\,\langle D\rangle^s g
\]
and bounding 
\[
\begin{aligned}
  \|[\langle D\rangle^s,V_f]g\|_{L^2}
  &\;\le\;\|\langle D\rangle^s(fg)\|_{L^2}\;+\;\|f\,\langle D\rangle^s g\|_{L^2} && \\
  &\;\le\;C\bigl(\|f\|_{H^s}\|g\|_{L^\infty}+\|f\|_{L^\infty}\|g\|_{H^s}\bigr)\;+\;\|f\|_{L^\infty}\|g\|_{H^s}, && 
\end{aligned}
\]
where $\|\langle D\rangle^s(fg)\|_{L^2}$ was bounded using (i) and $\|f\,\langle D\rangle^s g\|_{L^2}$ by  $(2,\infty)$ H\"older.

\begin{lemma}[Lower semicontinuity of $\|\nabla\,\cdot\,\|_{L^2}$ under strong $L^2$ convergence]\label{lem:LSC}
Let $\{u_n\}_{n\ge 1}\subset H^1(\T)$ satisfy
\begin{enumerate}
\item[\textup{(i)}] $u_n\to u$ in $L^2(\T);$
\item[\textup{(ii)}] $\sup_{n\ge 1}\|\nabla u_n\|_{L^2}<\infty.$
\end{enumerate}
Then $u\in H^1(\T)$ and
\begin{equation}\label{eq:LSC}
  \|\nabla u\|_{L^2}^2 \;\le\; \liminf_{n\to\infty}\|\nabla u_n\|_{L^2}^2.
\end{equation}
\end{lemma}
\noindent See, e.g., \cite[Prop.~3.5(iii)]{Brezis2011}.

\begin{lemma}[1D Gagliardo--Nirenberg on $\T$]\label{lem:GN1d}
For every $u\in H^1(\T),$
\begin{equation}\label{eq:Linfty_bound}
  \|u\|_{L^\infty(\T)}^{\,2}
  \;\le\; \frac{1}{|\T|}\,\|u\|_{L^2}^2 + 2\,\|u\|_{L^2}\,\|\nabla u\|_{L^2},
\end{equation}
and consequently
\begin{equation}\label{eq:GN1d}
  \|u\|_{L^4(\T)}^{\,4}
  \;\le\; \frac{1}{|\T|}\,\|u\|_{L^2}^4 + 2\,\|u\|_{L^2}^3\,\|\nabla u\|_{L^2}.
\end{equation}
\end{lemma}

\begin{proof}
This is a standard result (see e.g.~\cite[\S 8.2]{Brezis2011}); we go briefly through the proof to fix the explicit constants.

In $d=1,$ $H^1(\T)\hookrightarrow C(\T),$ so $u$ admits a continuous representative and $|u|^2\in H^1(\T)$ with $(|u|^2)'=2\,\mathrm{Re}(\overline{u}\,\nabla u)$ a.e. Pick $x_0\in\T$ with $|u(x_0)|^2=\min_\T|u|^2\le\tfrac{1}{|\T|}\|u\|_{L^2}^2.$ For any $x\in\T,$ integrating $(|u|^2)'$  from $x_0$ to $x$ and bounding by the full integral over $\T,$
\[
  |u(x)|^2 - |u(x_0)|^2
  \;\le\; \int_\T |(|u|^2)'|\,dy
  \;=\; 2\int_\T |\mathrm{Re}(\overline{u}\,\nabla u)|\,dy
  \;\le\; 2\,\|u\|_{L^2}\,\|\nabla u\|_{L^2},
\]
by Cauchy--Schwarz in the last step. Taking the supremum in $x$ yields~\eqref{eq:Linfty_bound}; multiplying by $\|u\|_{L^2}^2$ and using $\|u\|_{L^4}^4\le\|u\|_{L^\infty}^2\,\|u\|_{L^2}^2$ gives~\eqref{eq:GN1d}.
\end{proof}

\medskip

Since the free evolution~\eqref{eq:free} is a unitary conjugation, $S(t)$ preserves every Schatten norm: $\|S(t)\gamma\|_{\Sp^r}=\|\gamma\|_{\Sp^r}$ for all $r$ and all $t$. Using this fact and the subadditivity of the norm $\|\int f dt\| \leq \int \|f\|dt,$ we have the following

\begin{lemma}[Duhamel estimates]
\label{lem:Duhamel}
For every $T>0$:
\begin{align}
  \|S(t)\gamma_0\|_{L^\infty([0,T];\,H^s\Sp^r)}
  &= \|\gamma_0\|_{H^s\Sp^r},
  \label{eq:free_iso}\\[4pt]
  \biggl\|\int_0^t S(t-\tau)F(\tau)\,d\tau
  \biggr\|_{L^\infty([0,T];\,H^s\Sp^1)}
  &\le T\,
  \|F\|_{L^\infty([0,T];\,H^s\Sp^1)}.
  \label{eq:Duh_Linfty}
\end{align}
\end{lemma}

\begin{remark}[Role of non-negativity of the operator]\label{rem:positivity}
The non-negativity assumption $\gamma\ge 0$ is used structurally in several places in this paper, including:
\begin{itemize}
\item the identification
$\|\gamma\|_{H^1\Sp^1} = \tr\gamma + \tr(-\Delta\gamma)$, used in the finiteness of the energy (Lemma~\ref{lem:E_finite}) and in the $H^1\Sp^1$ a priori bound (Lemma~\ref{lem:finite_system}(b));
\item the orbital-level Cauchy--Schwarz and Hoffmann--Ostenhof chain inside the proof of Lemma~\ref{lem:finite_NLS} (Step~4), which closes the focusing $H^1$ a priori bound at the smooth approximant level and is then lifted to $\gamma$ by Lemma~\ref{lem:finite_system}(b).
\end{itemize}
Each of these relies on the spectral identity
\begin{equation}\label{eq:spectral_identity}
  \|\gamma\|_{H^s\Sp^1} = \sum_j \lambda_j\,\|\phi_j\|_{H^s}^2,\qquad \gamma\ge 0.
\end{equation}
Indeed, decomposing $\gamma=\sum_j\lambda_j|\phi_j\rangle\langle\phi_j|$ with $\lambda_j\ge 0$ and $\{\phi_j\}$ orthonormal gives $\langle D\rangle^s\gamma\langle D\rangle^s = \sum_j\lambda_j|\langle D\rangle^s\phi_j\rangle\langle\langle D\rangle^s\phi_j|\ge 0$; its $\Sp^1$-norm therefore equals its trace, $\sum_j\lambda_j\|\langle D\rangle^s\phi_j\|_{L^2}^2 = \sum_j\lambda_j\|\phi_j\|_{H^s}^2$. (No cyclicity of the trace involving the unbounded $\langle D\rangle^s$ is invoked: non-negativity carries the identification, and the assumption $\gamma\in H^s\Sp^1$ forces $\phi_j\in H^s$ whenever $\lambda_j>0$, making the series convergent.) Specializing to $s=1$ and expanding $\|\langle D\rangle\phi_j\|_{L^2}^2 = \|\phi_j\|_{L^2}^2 + \|\nabla\phi_j\|_{L^2}^2$ gives the equivalent form
\begin{equation}\label{eq:H1S1_split}
  \|\gamma\|_{H^1\Sp^1} = \tr\gamma + \tr(-\Delta\gamma),
\end{equation}
with $\tr(-\Delta\gamma):=\sum_j\lambda_j\|\nabla\phi_j\|_{L^2}^2 = \|\nabla\sqrt{\gamma}\|_{\Sp^2}^2 \in [0,\infty]$ (cf.\ also~\cite
{LewinSabin2015}). We adopt the shorthands
\begin{equation}\label{eq:Kcal_def}
\begin{aligned}
  \mathcal{M}(\gamma)\;&:=\;\|\gamma\|_{\Sp^1}\;=\;\sum_j\lambda_j,\\
  \mathcal{K}(\gamma)\;&:=\;\tr(-\Delta\gamma)\;=\;\sum_j\lambda_j\|\nabla\phi_j\|_{L^2}^2,
\end{aligned}
\end{equation}
for the \emph{mass} and the \emph{kinetic energy} of $\gamma\ge 0,$ respectively; for a solution $\gamma(t)$ of~\eqref{eq:Alber_op} we write $\mathcal{M}(t):=\mathcal{M}(\gamma(t))$ and $\mathcal{K}(t):=\mathcal{K}(\gamma(t)).$ The same notation is reused for finite-rank approximants. 

At the density level, the spectral representation becomes
\begin{equation}\label{eq:rho_spectral}
  \rho_\gamma(x) \;=\; K_\gamma(x,x) \;=\; \sum_j\lambda_j\,|\phi_j(x)|^2 \;\ge\; 0,
  \qquad x\in\T.
\end{equation}
In our setting of $\gamma \in H^s\Sp^1,$ the series \eqref{eq:rho_spectral} converges  uniformly on $\T.$

We also note the fact that self-adjointness and non-negativity propagate along the flow of \eqref{eq:Alber_op} (cf. Lemma~\ref{lem:finite_system}(a)), so the assumption $\gamma_0=\gamma_0^*\ge 0$ implies $\gamma(t)=\gamma(t)^*\ge 0$ for all $t$ in the interval of existence.
\end{remark}

\subsection{Density and nonlinear estimates}\label{sec:nonlin}

We collect here the key density and nonlinearity estimates, connecting the operator $\gamma$ to its spatial density $\rho_\gamma(x)=K_\gamma(x,x)$.

\begin{lemma}[Bessel-kernel estimate]
\label{lem:Bessel}
Let $s>1/2$ and $d=1$. For every $\gamma\in H^s\Sp^1(\T)$,
\begin{equation}\label{eq:Bessel}
  \|\rho_\gamma\|_{L^\infty(\T)}
  \le B_s\,\|\gamma\|_{H^s\Sp^1},
\end{equation}
where $B_s = (2\pi)^{-1}\sum_{n\in\Z} \langle n\rangle^{-2s}<\infty$.
\end{lemma}

\begin{proof}
Set $\tilde\gamma:=\langle D\rangle^s\gamma\langle D\rangle^s\in\Sp^1$, so that $\gamma = \langle D\rangle^{-s}\tilde\gamma\langle D\rangle^{-s}$ and
\[
  \|\tilde\gamma\|_{\mathrm{op}}\le\|\tilde\gamma\|_{\Sp^1}=\|\gamma\|_{H^s\Sp^1}.
\]

Moreover, we will use the following {\em claim}:
\begin{equation}\label{eq:Bessel_kernel_norm}
  \|\langle D\rangle^{-s}\delta_{x_0}\|_{L^2}^2 = \|\delta_{x_0}\|_{H^{-s}}^2 = B_s
\end{equation}
which will be proved below.

Denote $G_{x_0}:=\langle D\rangle^{-s}\delta_{x_0}$ the {\em Bessel potential kernel}, which according to the claim is an $L^2$ function. Then
\begin{align*}
  \langle G_{x_0} ,\tilde\gamma G_{x_0} \rangle
  &= \langle\langle D\rangle^{-s}\delta_{x_0} ,\,\tilde\gamma \langle D\rangle^{-s}\delta_{x_0} \rangle \\
  &= \langle\delta_{x_0} ,\,\langle D\rangle^{-s}\tilde\gamma \langle D\rangle^{-s}\delta_{x_0} \rangle \\
  &= \langle\delta_{x_0} ,\, \gamma\delta_{x_0} \rangle = \rho_\gamma(x_0).
\end{align*}
Thus,
\[
  |\rho_\gamma(x_0)|
  \le\|\tilde\gamma\|_{\mathrm{op}}\,\|\langle D\rangle^{-s}\delta_{x_0}\|_{L^2}^2
  \le B_s\,\|\gamma\|_{H^s\Sp^1} \qquad \forall x_0 \in \T.
\]

\noindent {\em Proof of the claim \eqref{eq:Bessel_kernel_norm}:}
Since $\widehat{\delta_{x_0}}(n)=(2\pi)^{-1/2}e^{-inx_0}$,
\[
  \|\delta_{x_0}\|_{H^{-s}(\T)}^2
  = \sum_{n\in\Z}\langle n\rangle^{-2s}|\widehat{\delta_{x_0}}(n)|^2
  = \frac{1}{2\pi}\sum_{n\in\Z}\langle n\rangle^{-2s}
  = B_s,
\]
which is finite (since $s>1/2$).
Moreover, since $\langle D\rangle^{-s}$ is the Fourier multiplier with symbol $\langle n\rangle^{-s},$  the distribution $G_{x_0}=\langle D\rangle^{-s}\delta_{x_0}$ has Fourier coefficients
\[
  \widehat{G_{x_0}}(n)
  = \frac{1}{\sqrt{2\pi}}\,\langle n\rangle^{-s}\,e^{-inx_0}.
\]
By Parseval,
\[
\| G_{x_0}\|_{L^2(\T)}^2 =  \frac{1}{2\pi}\sum_{n\in\Z}\langle n\rangle^{-2s} = B_s.
\]

\end{proof}

\begin{lemma}[Finiteness of the energy on $H^1\Sp^1$]\label{lem:E_finite}
Let $\gamma\in H^1\Sp^1(\T)$ be self-adjoint and non-negative. Then the energy $E(\gamma)$ defined in~\eqref{eq:E_def} is finite, with
\begin{equation}\label{eq:E_finite_bound}
  |E(\gamma)|\;\le\;|p|\,\|\gamma\|_{H^1\Sp^1}\;+\;\frac{|q|}{2}\,B_1\,\|\gamma\|_{\Sp^1}\|\gamma\|_{H^1\Sp^1},
\end{equation}
where $B_1$ is the Bessel-kernel constant of Lemma~\ref{lem:Bessel} at $s=1.$
\end{lemma}

\begin{proof}
For self-adjoint non-negative $\gamma,$ the spectral identity~\eqref{eq:H1S1_split} gives $\tr(-\Delta\gamma)\le\|\gamma\|_{H^1\Sp^1},$ hence $|p\,\tr(\Delta\gamma)|\le|p|\,\|\gamma\|_{H^1\Sp^1}.$ For the potential term, $\rho_\gamma\ge 0$ and $\int_\T\rho_\gamma=\tr\gamma=\|\gamma\|_{\Sp^1},$ so
\[
  \|\rho_\gamma\|_{L^2}^2
  \le\|\rho_\gamma\|_{L^1}\|\rho_\gamma\|_{L^\infty}
  \le B_1\,\|\gamma\|_{\Sp^1}\|\gamma\|_{H^1\Sp^1},
\]
using Lemma~\ref{lem:Bessel} at $s=1.$
\end{proof}

\begin{lemma}[$H^s$ trace estimate]\label{lem:trace}
Let $s>1/2$. For any $\gamma\in H^s\Sp^1(\T)$,
\begin{equation}\label{eq:trace_Hs}
  \|\rho_\gamma\|_{H^s(\T)}
  \le C\,\|\gamma\|_{H^s\Sp^1}.
\end{equation}
\end{lemma}

\begin{proof}
A simple proof based on the spectral decomposition of $\gamma$ is possible if $\gamma$ is self-adjoint and non-negative. However, we also want to use this estimate for arbitrary differences  of non-negative self-adjoint operators $\gamma_1-\gamma_2.$ In order to accommodate that, we give a proof by duality that requires neither self-adjointness nor non-negativity. The argument uses only the Schatten--H\"older inequality and the fact that $H^s(\T)$ is an algebra for $s>1/2.$

\medskip

Since $s>1/2$, $H^s(\T)$ is a Banach algebra under pointwise multiplication: $\|hg\|_{H^s}\le C\|h\|_{H^s}\|g\|_{H^s}$. By duality, for $f\in H^{-s}(\T)$ and $h\in H^s(\T)$,
\begin{equation}\label{eq:trace_mult}
  \|fh\|_{H^{-s}}
  \le C\|f\|_{H^{-s}}\|h\|_{H^s}.
\end{equation}

\medskip

Moreover, for $g\in L^2(\T),$ equation \eqref{eq:trace_mult} gives
\begin{align*}
  \|\langle D\rangle^{-s}V_f\langle D\rangle^{-s}g\|_{L^2}
  &= \|V_f(\langle D\rangle^{-s}g)\|_{H^{-s}} \\
  &\le C\|f\|_{H^{-s}}\|\langle D\rangle^{-s}g\|_{H^s}
  = C\|f\|_{H^{-s}}\|g\|_{L^2},
\end{align*}
that is,
\begin{equation}\label{eq:trace_op_bound}
  \|\langle D\rangle^{-s}V_f\langle D\rangle^{-s}\|_{\mathrm{op}}
  \le C\|f\|_{H^{-s}}.
\end{equation}

\medskip

Now, for $f\in C^\infty(\T)$ the multiplication operator $V_f$ is bounded and $V_f\gamma\in\Sp^1$; pairing the density $\rho_\gamma(x)=K_\gamma(x,x)$ against $f$ gives
\[
  \int_\T f(x)\,\rho_\gamma(x)\,dx
  = \int_\T f(x)\,K_\gamma(x,x)\,dx
  = \tr(V_f\,\gamma).
\]
By cyclicity of the trace,
\[
\begin{aligned}
  \tr(V_f\gamma)
  &= \tr\bigl((\langle D\rangle^{-s}V_f\langle D\rangle^{-s})\,
  (\langle D\rangle^s\gamma\langle D\rangle^s)\bigr) \implies \\
    |\tr(V_f\gamma)|\;&\le\;\|\langle D\rangle^{-s}V_f\langle D\rangle^{-s}\|_{\mathrm{op}}\,\|\langle D\rangle^s\gamma\langle D\rangle^s\|_{\Sp^1}
\end{aligned}
\]
and the observation $\|\langle D\rangle^s\gamma\langle D\rangle^s\|_{\Sp^1}=\|\gamma\|_{H^s\Sp^1}$
combined with~\eqref{eq:trace_op_bound} yields
\begin{align*}
\Bigl|\int_\T f\,\rho_\gamma\,dx\Bigr|&=  \tr(V_f\gamma)
  \le \|\langle D\rangle^{-s}V_f\langle D\rangle^{-s}\|_{\mathrm{op}}\,
  \|\gamma\|_{H^s\Sp^1}\\
  &\le C\|f\|_{H^{-s}}\|\gamma\|_{H^s\Sp^1}
  \qquad\text{for all }f\in C^\infty(\T).
\end{align*}
Since $C^\infty(\T)$ is dense in $H^{-s}(\T)$, the result follows.
\end{proof}

\begin{lemma}[Conjugation bound, $\Sp^\infty$]\label{lem:conj}
For $s>1/2$, $d=1$, and $f\in H^s(\T)$,
\begin{equation}\label{eq:conj}
  \|\langle D\rangle^s V_f
  \langle D\rangle^{-s}\|_{\mathrm{op}}
  \le C\|f\|_{H^s}.
\end{equation}
\end{lemma}

\begin{proof}
Write $\langle D\rangle^s V_f \langle D\rangle^{-s} = V_f + [\langle D\rangle^s,V_f] \langle D\rangle^{-s}$. The first term: $\|V_f\|_{\mathrm{op}}=\|f\|_{L^\infty} \le C\|f\|_{H^s}.$ For the commutator, let $h\in L^2(\T)$ and set $g=\langle D\rangle^{-s}h.$ Inequality~\eqref{eq:KP_commutator} gives
\[
  \|[\langle D\rangle^s,V_f]g\|_{L^2}
  \;\le\; C\bigl(\|f\|_{H^s}\|g\|_{L^\infty}+\|f\|_{L^\infty}\|g\|_{H^s}\bigr).
\]
Using $\|f\|_{L^\infty}\le C\|f\|_{H^s}$ (Sobolev) and $\|g\|_{H^s}=\|h\|_{L^2},$
\[
  \|[\langle D\rangle^s,V_f]g\|_{L^2}
  \;\le\; C\|f\|_{H^s}\bigl(\|g\|_{L^\infty}+\|h\|_{L^2}\bigr).
\]
Since $g=\langle D\rangle^{-s}h,$ Sobolev embedding gives 
\[
\|g\|_{L^\infty}\le C\|g\|_{H^{1/2+\varepsilon}} = C\|h\|_{H^{1/2+\varepsilon-s}}\le C\|h\|_{L^2}
\]
 for $\varepsilon$ small enough that $1/2+\varepsilon-s\le 0.$ Hence
\[
  \|[\langle D\rangle^s,V_f]\,\langle D\rangle^{-s}h\|_{L^2}
  \;\le\; C\|f\|_{H^s}\|h\|_{L^2}.
\]
\end{proof}

Now we are in a position to prove the following

\begin{proposition}[Bilinear bound]
\label{prop:bilinear}
Let $s>1/2$, $d=1$. For $\gamma_1,\gamma_2\in H^s\Sp^1(\T)$,
\begin{equation}\label{eq:bilinear}
  \|[V_{\rho_{\gamma_1}},\gamma_2]\|_{H^s\Sp^1}
  \le C\,\|\gamma_1\|_{H^s\Sp^1}
  \,\|\gamma_2\|_{H^s\Sp^1}.
\end{equation}
In particular, the nonlinearity $\mathcal{N}(\gamma)=[V_{\rho_\gamma},\gamma]$ satisfies
\begin{equation}\label{eq:nonlin_bound}
  \|\mathcal{N}(\gamma)\|_{H^s\Sp^1}
  \le C\,\|\gamma\|_{H^s\Sp^1}^2.
\end{equation}
\end{proposition}

\begin{proof}

Denote for brevity $\rho=\rho_{\gamma_1}.$ By the triangle inequality,
\[
  \|[V_{\rho},\gamma_2]\|_{H^s\Sp^1}
  \le \|V_{\rho}\gamma_2\|_{H^s\Sp^1}
  + \|\gamma_2 V_{\rho}\|_{H^s\Sp^1}.
\]
We detail the left-multiplication term in full; the right-multiplication term is reduced to it by taking adjoints.

By definition,
\begin{align}\label{eq:bil_def}
  \|V_\rho\gamma_2\|_{H^s\Sp^1}
  &= \|\langle D\rangle^s V_\rho\gamma_2
  \langle D\rangle^s\|_{\Sp^1}\notag\\
  &= \|(\langle D\rangle^s V_\rho
  \langle D\rangle^{-s})
  (\langle D\rangle^s\gamma_2
  \langle D\rangle^s)\|_{\Sp^1}.
\end{align}
By the Schatten H\"older inequality $\Sp^\infty\cdot\Sp^1\to\Sp^1$:
\begin{equation}\label{eq:bil_holder}
  \|(\langle D\rangle^s V_\rho
  \langle D\rangle^{-s})
  (\langle D\rangle^s\gamma_2
  \langle D\rangle^s)\|_{\Sp^1}
  \le \|\langle D\rangle^s V_\rho
  \langle D\rangle^{-s}\|_{\mathrm{op}}
  \,\|\gamma_2\|_{H^s\Sp^1}.
\end{equation}
By virtue of  Lemma~\ref{lem:conj},
\begin{equation}\label{eq:bil_conj}
  \|\langle D\rangle^s V_\rho
  \langle D\rangle^{-s}\|_{\mathrm{op}}
  \le C\|\rho\|_{H^s}.
\end{equation}
Moreover, by virtue of Lemma~\ref{lem:trace},
\begin{equation}\label{eq:bil_trace}
  \|\rho\|_{H^s}
  \le C\|\gamma_1\|_{H^s\Sp^1}.
\end{equation}
Combining~\eqref{eq:bil_def}--\eqref{eq:bil_trace} bounds $\|V_\rho\gamma_2\|_{H^s\Sp^1}$ by $C\|\gamma_1\|_{H^s\Sp^1}\|\gamma_2\|_{H^s\Sp^1}.$

\end{proof}

\section{Local well-posedness and Fourier--Galerkin approximation}
\label{sec:galerkin}

This section collects the machinery used to manipulate $H^1\Sp^1$ solutions of~\eqref{eq:Alber_op} and their approximating sequences in what follows. Subsection~\ref{subsec:lwp} sets up the mild formulation and the short-time $H^1\Sp^1$ local well-posedness (Proposition~\ref{prop:local}). Subsections~\ref{subsec:galerkin_static}--\ref{subsec:galerkin_dynamic} develop the Fourier--Galerkin approximation: the static truncation $\gamma^{(N)}:=P_N\gamma P_N$ (Proposition~\ref{prop:galerkin}), the smooth orbital engine (Lemma~\ref{lem:finite_NLS}), and the bridge lemma (Lemma~\ref{lem:finite_system}) which lifts key qualitative properties and bounds from the smooth approximant to any $H^1\Sp^1$ solution.

\subsection{Mild form and local well-posedness}
\label{subsec:lwp}

We write~\eqref{eq:Alber_op} in mild form
\begin{equation}\label{eq:Duhamel}
  \gamma(t) = S(t)\gamma_0
  - iq\int_0^t S(t-\tau)\,
  \mathcal{N}(\gamma(\tau))\,d\tau,
\end{equation}
with $\mathcal{N}(\gamma)=[V_{\rho_\gamma},\gamma]$.

\medskip

Let $T>0$ and define the Banach space
\begin{equation}\label{eq:XT}
  X_T = L^\infty([0,T];\,H^1\Sp^1),
\end{equation}
with norm $\|\gamma\|_{X_T} = \|\gamma\|_{L^\infty_t H^1\Sp^1}$. For $\gamma_0\in H^1\Sp^1$ and $\gamma\in X_T,$ define the Duhamel map $\Phi=\Phi_{\gamma_0}$ by
\begin{equation}\label{eq:Duhamel_map}
  \Phi(\gamma)(t) \;:=\; S(t)\gamma_0
  - iq\int_0^t S(t-\tau)\,\mathcal{N}(\gamma(\tau))\,d\tau,
  \qquad t\in[0,T].
\end{equation}

\begin{proposition}[Local well-posedness on $H^1\Sp^1$]
\label{prop:local}
For every $\delta>0$ there exists $T>0$ such that if $\|\gamma_0\|_{H^1\Sp^1}\le\delta$, then the Duhamel map $\Phi_{\gamma_0}$ is a contraction on $\bigl\{\gamma\in X_T : \|\gamma\|_{X_T}\le 2\delta\bigr\}$.

Moreover, the solution depends Lipschitz-continuously on the initial datum: if $\|\gamma_0^{(1)}\|_{H^1\Sp^1}, \|\gamma_0^{(2)}\|_{H^1\Sp^1}\le\delta$ and $\gamma^{(1)},\gamma^{(2)}$ are the corresponding solutions, then
\begin{equation}\label{eq:lwp_stab}
  \|\gamma^{(1)}-\gamma^{(2)}\|_{X_T}
  \le 2\,\|\gamma_0^{(1)}-\gamma_0^{(2)}\|_{H^1\Sp^1}.
\end{equation}
\end{proposition}

\begin{proof}
Write $\delta=\|\gamma_0\|_{H^1\Sp^1}$.

\medskip

\noindent\textbf{Step 1: Linear estimate.} By isometry of the free evolution,
\begin{equation}\label{eq:linear_XT}
  \|S(t)\gamma_0\|_{X_T}
  = \|\gamma_0\|_{H^1\Sp^1}
  = \delta.
\end{equation}

\medskip

\noindent\textbf{Step 2: Nonlinear estimate.} By Proposition~\ref{prop:bilinear} and the Duhamel estimate~\eqref{eq:Duh_Linfty}:
\begin{equation}\label{eq:Duh_S1}
  \biggl\|\int_0^t S(t-\tau)
  \mathcal{N}(\gamma)\,d\tau
  \biggr\|_{X_T}
  \le C\,T\,\|\gamma\|_{X_T}^2.
\end{equation}

\noindent\textbf{Step 3: Closing.} On the ball $\|\gamma\|_{X_T}\le 2\delta$:
\[
  \|\Phi(\gamma)\|_{X_T}
  \le \delta + C\,T\,(2\delta)^2
  = \delta\,(1 + 4C\,T\,\delta).
\]
Choosing $T$ small enough that $4C\,T\,\delta\le 1/2$, i.e.\
\begin{equation}\label{eq:delta_T}
  T \le \frac{1}{8C\,\delta},
\end{equation}
gives $\|\Phi(\gamma)\|_{X_T}\le\tfrac{3}{2}\delta<2\delta$, so $\Phi$ maps the ball $\|\gamma\|_{X_T}\le 2\delta$ to itself.

For the contraction, let $\gamma_1,\gamma_2$ lie in the ball. Using the linearity $\rho_{\gamma_1}-\rho_{\gamma_2}=\rho_{\gamma_1-\gamma_2}$, polarize:
\[
  \mathcal{N}(\gamma_1)-\mathcal{N}(\gamma_2)
  = [V_{\rho_{\gamma_1-\gamma_2}},\gamma_1]
  + [V_{\rho_{\gamma_2}},\gamma_1-\gamma_2].
\]
Applying Proposition~\ref{prop:bilinear} to each term gives
\[
  \|\mathcal{N}(\gamma_1)-\mathcal{N}(\gamma_2)\|_{H^1\Sp^1}
  \le C\bigl(\|\gamma_1\|_{H^1\Sp^1}+\|\gamma_2\|_{H^1\Sp^1}\bigr)
  \,\|\gamma_1-\gamma_2\|_{H^1\Sp^1},
\]
and the same Duhamel estimate as in Step~2 yields
\[
  \|\Phi(\gamma_1)-\Phi(\gamma_2)\|_{X_T}
  \le 4C\,T\,\delta\,
  \|\gamma_1-\gamma_2\|_{X_T}
  \le \tfrac{1}{2}\,
  \|\gamma_1-\gamma_2\|_{X_T},
\]
so $\Phi$ is a contraction.

\medskip

\noindent\textbf{Step 4: Continuous dependence.} Let $\gamma^{(j)}=\Phi_{\gamma_0^{(j)}}(\gamma^{(j)})$ for $j=1,2,$ with $\Phi_{\gamma_0}$ as in~\eqref{eq:Duhamel_map}. Subtracting and using the linear estimate together with the Lipschitz bound on $\mathcal{N}$ from Step~3,
\[
  \|\gamma^{(1)}-\gamma^{(2)}\|_{X_T}
  \le \|\gamma_0^{(1)}-\gamma_0^{(2)}\|_{H^1\Sp^1}
  + \tfrac12
    \|\gamma^{(1)}-\gamma^{(2)}\|_{X_T},
\]
yielding~\eqref{eq:lwp_stab}.
\end{proof}

\begin{remark}\label{rem:delta_T}
The proof gives the explicit lifespan $T(\delta)=1/(8C\delta)$, so data of any size can be handled on a sufficiently short time interval. Equivalently, writing $\delta(T):=1/(8CT),$ the contraction holds on $[0,T]$ whenever $\|\gamma_0\|_{H^1\Sp^1}\le\delta(T).$
\end{remark}

\subsection{Static Fourier truncation}
\label{subsec:galerkin_static}

\begin{proposition}[Fourier--Galerkin approximation in $H^s\Sp^1$]
\label{prop:galerkin}
Let $s\ge 0.$ For $N\in\mathbb{N},$ let $P_N$ denote the orthogonal projection in $L^2(\T)$ onto $\mathrm{span}\{e^{inx}:|n|\le N\},$ and define the Fourier--Galerkin truncation
\begin{equation}\label{eq:galerkin_def}
  \gamma^{(N)}\;:=\;P_N\gamma P_N.
\end{equation}
Then for every $\gamma\in H^s\Sp^1(\T)$:
\begin{enumerate}
\item[\textup{(i)}] \emph{(Regularization.)} $\gamma^{(N)}$ is finite-rank with range and corange contained in $P_NL^2(\T);$ in particular $\gamma^{(N)}\in H^\sigma\Sp^1$ for every $\sigma\ge 0.$
\item[\textup{(ii)}] \emph{(Inheritance.)} If $\gamma=\gamma^{*},$ then $\gamma^{(N)}=(\gamma^{(N)})^{*}.$ If in addition $\gamma\ge 0,$ then $\gamma^{(N)}\ge 0.$
\item[\textup{(iii)}] \emph{(Monotone norm bound.)} $\|\gamma^{(N)}\|_{H^s\Sp^1}\le\|\gamma\|_{H^s\Sp^1}$ for every $N,$ and the map $N\mapsto\|\gamma^{(N)}\|_{H^s\Sp^1}$ is nondecreasing.
\item[\textup{(iv)}] \emph{(Strong convergence.)} $\gamma^{(N)}\to\gamma$ in $H^s\Sp^1$ as $N\to\infty.$
\item[\textup{(v)}] \emph{(Density-level consequences.)} Let $\rho^{(N)}:=\rho_{\gamma^{(N)}} \ge 0.$ Then $\rho^{(N)}\to\rho_\gamma$ in $L^1(\T),$ and if $\gamma\in H^\sigma\Sp^1$ for some $\sigma>1/2$ then additionally $\rho^{(N)}\to\rho_\gamma$ in $H^\sigma(\T).$
\item[\textup{(vi)}] \emph{(Finite spectral representation.)} If $\gamma=\gamma^{*}\ge 0,$ then $\gamma^{(N)}$ admits a finite spectral decomposition
\begin{equation}\label{eq:galerkin_spec}
  \gamma^{(N)}
  \;=\; \sum_{k=1}^{M_N}\mu_k^{(N)}\,
  |\psi_k^{(N)}\rangle\langle\psi_k^{(N)}|,
  \qquad M_N\le 2N+1,
\end{equation}
with $\mu_k^{(N)}\ge 0$ and $\{\psi_k^{(N)}\}_{k=1}^{M_N}$ orthonormal in $L^2(\T);$ each $\psi_k^{(N)}\in P_NL^2(\T)$ is a trigonometric polynomial of degree at most $N,$ and in particular $\psi_k^{(N)}\in C^\infty(\T).$
\end{enumerate}
\end{proposition}

\begin{proof}
(i) Both $P_N$ and $P_N\gamma P_N$ factor through the $(2N+1)$-dimensional space $P_NL^2,$ so $\gamma^{(N)}$ has rank $\le 2N+1.$ Since $[P_N,\langle D\rangle^\sigma]=0$ for every $\sigma\ge 0,$
\[
  \langle D\rangle^\sigma\gamma^{(N)}\langle D\rangle^\sigma
  =P_N\bigl(\langle D\rangle^\sigma\gamma
  \langle D\rangle^\sigma\bigr)P_N,
\]
hence $\gamma^{(N)}\in H^\sigma\Sp^1$ with $\|\gamma^{(N)}\|_{H^\sigma\Sp^1}\le\|\gamma\|_{H^\sigma\Sp^1}$ by the Schatten contraction $\|P_NAP_N\|_{\Sp^1}\le\|A\|_{\Sp^1}$ (cf.\ Proposition~\ref{prop:schatten_holder} with $r_1=r_2=\infty$ and $\|P_N\|_{\mathrm{op}}=1$).

(ii) $(\gamma^{(N)})^{*}=P_N\gamma^{*}P_N=\gamma^{(N)},$ and $\langle\varphi,\gamma^{(N)}\varphi\rangle=\langle P_N\varphi,\gamma\,P_N\varphi\rangle\ge 0$ for every $\varphi\in L^2.$

(iii) The norm bound is the estimate of (i) at $\sigma=s;$ monotonicity follows from $P_{N+1}P_N=P_NP_{N+1}=P_N,$ which yields $P_N\gamma^{(N+1)}P_N=\gamma^{(N)}$ and hence $\|\gamma^{(N)}\|_{H^s\Sp^1}\le\|\gamma^{(N+1)}\|_{H^s\Sp^1}$ by the same Schatten contraction.

(iv) Write $\gamma-\gamma^{(N)}=(I-P_N)\gamma+P_N\gamma(I-P_N).$ It suffices to show $(I-P_N)A\to 0$ and $A(I-P_N)\to 0$ in $\Sp^1$ for any $A\in\Sp^1$, applied to $A=\langle D\rangle^s\gamma\langle D\rangle^s.$ Take a singular-value decomposition $A=\sum_n s_n|u_n\rangle\langle v_n|$ with $\{u_n\},\{v_n\}$ orthonormal in $L^2$ and $\sum_n s_n=\|A\|_{\Sp^1}<\infty.$ Then
\[
  \|(I-P_N)A\|_{\Sp^1}
  \;\le\; \sum_n s_n\,\|(I-P_N)u_n\|_{L^2},
\]
and since $P_N\to I$ strongly on $L^2$ each summand vanishes pointwise in $n,$ while the bound $\|(I-P_N)u_n\|_{L^2}\le 1$ together with $\{s_n\}\in\ell^1$ supplies the dominant; dominated convergence on the $n$-sum gives $\|(I-P_N)A\|_{\Sp^1}\to 0.$ The argument for $A(I-P_N)$ is identical, swapping the roles of $\{u_n\},\{v_n\}.$ Combining, $\|\gamma-\gamma^{(N)}\|_{H^s\Sp^1}\to 0.$

(v) The kernel-to-density map $\gamma\mapsto\rho_\gamma$ is bounded $\Sp^1\to L^1(\T)$ with $\|\rho_\gamma\|_{L^1}\le\|\gamma\|_{\Sp^1}$ (dual pairing $\langle\varphi,\rho_\gamma\rangle=\tr(V_\varphi\gamma)$ with $\|V_\varphi\|_{\mathrm{op}}=\|\varphi\|_{L^\infty}$). By Lemma~\ref{lem:trace} it is also bounded $H^\sigma\Sp^1\to H^\sigma(\T)$ for every $\sigma>1/2.$ Applying these to the difference $\gamma^{(N)}-\gamma$ and using (iv) gives the two convergences.

Non-negativity follows immediately from the spectral decomposition, (vi).

(vi) $\gamma^{(N)}$ is a self-adjoint non-negative operator whose range lies in $P_NL^2\cong\C^{2N+1};$ the finite-dimensional spectral theorem yields~\eqref{eq:galerkin_spec}, and the eigenvectors $\psi_k^{(N)}\in P_NL^2$ are by construction trigonometric polynomials of degree $\le N,$ hence smooth.
\end{proof}

%
%

\subsection{Dynamics of the finite-rank approximant}
\label{subsec:galerkin_dynamic}

We now record how the static Fourier--Galerkin truncation behaves under the nonlinear flow. The construction proceeds in two steps. First, Lemma~\ref{lem:finite_NLS} solves the orbital NLS system on $\T$ with finite-rank smooth initial data, producing globally smooth orbitals and the associated finite-rank operator $\gamma$ self-adjoint and non-negative by construction. Second, Lemma~\ref{lem:finite_system} applies this to the Galerkin-truncated initial datum $\gamma_0^{(N)}=P_N\gamma_0 P_N$ and proves convergence to the $H^1\Sp^1$ solution of Proposition~\ref{prop:local}. Any such solution is shown to inherit  important qualitative properties (self-adjointness, positivity, energy bounds etc).

\begin{lemma}[Smooth orbital NLS system on $\T$]
\label{lem:finite_NLS}
Let $M\in\mathbb{N},$ $s\ge 0,$ let $(\mu_k)_{k=1}^M\subset[0,\infty),$ and let the functions $(\psi_k^0)_{k=1}^M\subset C^\infty(\T)$ be orthonormal in $L^2(\T).$ Then the orbital system
\begin{equation}\label{eq:Alber_system_finite_NLS}
\begin{aligned}
  i\partial_t\psi_k
  &=p\Delta\psi_k+qV_\rho\,\psi_k,\\
  \rho(x,t)
  &=\sum_{l=1}^{M}\mu_l\,|\psi_l(x,t)|^2,
\end{aligned}
\qquad k=1,\dots,M,
\end{equation}
with initial data $\psi_k(0)=\psi_k^0,$ admits a unique global solution
\[
  (\psi_k)_{k=1}^M\in C(\R;\,H^{s+2}(\T))^M\cap C^1(\R;\,H^s(\T))^M,
\]
preserving orthonormality $\langle\psi_k(t),\psi_l(t)\rangle=\delta_{kl}$ for all $k,l\in\{1,\dots,M\}$ and $t\in\R.$ Setting
\begin{equation}\label{eq:Alber_system_finite}
  \gamma(t)\;:=\;\sum_{k=1}^M\mu_k\,|\psi_k(t)\rangle\langle\psi_k(t)|,
\end{equation}
the operator satisfies $\gamma\in C(\R;H^{s+2}\Sp^1(\T))\cap C^1(\R;H^s\Sp^1(\T)),$ is self-adjoint and non-negative with $\operatorname{rank}\gamma(t)\le M,$ and solves the operator-valued Alber equation~\eqref{eq:Alber_op} with $\gamma(0)=\sum_k\mu_k|\psi_k^0\rangle\langle\psi_k^0|.$ Moreover the conservation laws
\begin{equation}\label{eq:finiteNLS_conservation}
\begin{aligned}
  &\|\gamma(t)\|_{\Sp^1}\;=\;\sum_{k=1}^M\mu_k\;=:\;\mathcal{M},
  \qquad
  \|\gamma(t)\|_{\Sp^2}^2\;=\;\sum_{k=1}^M\mu_k^{2},\\
  &E(\gamma(t))\;=\;E(\gamma(0))\;=:\;E_0
\end{aligned}
\end{equation}
hold for every $t\in\R,$ where $E$ is the energy~\eqref{eq:E_def}, and there exists $\bar Y=\bar Y(\|\gamma(0)\|_{H^1\Sp^1},p,q,|\T|)<\infty,$ monotone non-decreasing in its first argument, such that
\begin{equation}\label{eq:Ybar_finiteNLS_stmt}
  \|\gamma(t)\|_{H^1\Sp^1}\;=\;\mathcal{M}+\tr(-\Delta\gamma(t))\;\le\;\bar Y
  \qquad\text{for every }t\in\R.
\end{equation}
\end{lemma}

\begin{remark}\label{rem:barY}
Explicitly, in the focusing case ($pq>0$),
\[
  \bar Y\;=\;\mathcal{M}+ \frac{1}{4}\left(\frac{|q|\,\mathcal{M}^{3/2}}{|p|}\;+\;\sqrt{\frac{q^{2}\,\mathcal{M}^{3}}{p^{2}}\;+\;4\,\mathcal{K}(0)\;+\;\frac{2\,|q|\,\mathcal{M}^{2}}{|p|\,|\T|}}\right)^2.
\]
In the defocusing case ($pq<0$),
$
  \bar Y\;=\;\mathcal{M}+\mathcal{K}(0)+\frac{|q|}{2\,|p|}\,\|\rho_0\|_{L^2}^{\,2},
$
which by the Bessel-kernel estimate of Lemma~\ref{lem:Bessel} at $s=1$ further satisfies
\[
  \bar Y\;\le\;\|\gamma_0\|_{H^1\Sp^1}\Bigl(1+\frac{B_1\,|q|}{2\,|p|}\,\|\gamma_0\|_{H^1\Sp^1}\Bigr).
\]
See Step~4 of the proof for the derivation of these values.
\end{remark}

\begin{proof}
The overall structure of the argument follows the standard  template for nonlinear Schr\"odinger equations in $H^s$ with $s>d/2$ (cf.\ \cite{Cazenave2003}, Ch.\ 3-4) applied component-wise. The difference is in getting the a priori bound in the system.

\smallskip

\noindent\textbf{Step 1: Local existence.} For $s>0,$ $H^{s+2}(\T)$ is a Banach algebra, hence the nonlinearity
\[
  (\psi_l)_{l=1}^M\;\longmapsto\;\bigl(V_{\rho}\psi_k\bigr)_{k=1}^M
  \;=\;\Bigl(\Bigl(\sum_{l=1}^M\mu_l|\psi_l|^2\Bigr)\psi_k\Bigr)_{k=1}^M
\]
is  locally Lipschitz $H^{s+2}(\T)^M\to H^{s+2}(\T)^M.$ Denote for brevity $\Psi=(\psi_l)_{l=1}^M;$ we will show local existence in
\[
  X_{T_*}:=L^\infty([0,T_*];H^{s+2}(\T)^M),
  \qquad
  \|\Psi\|_{X_{T_*}}:=\max_{1\le k\le M}\|\psi_k\|_{L^\infty_t H^{s+2}_x}.
\]
Set 
\[
\mathcal{M}_0:=\sum_{l=1}^M\mu_l, \qquad \qquad R_0:=\max_k\|\psi_k^0\|_{H^{s+2}}
\]
(note that these are fixed  constants depending only on the size of the initial condition).
 The Duhamel map
\begin{equation}\label{eq:Phi_finiteNLS}
  \Phi(\Psi)_k(t)\;:=\;e^{-ipt\Delta}\psi_k^0\;-\;iq\int_0^t e^{-ip(t-\tau)\Delta}\,V_{\rho(\Psi(\tau))}\psi_k(\tau)\,d\tau,
\end{equation}
where $\rho(\Psi):=\sum_{l=1}^M\mu_l|\psi_l|^2,$
is the orbital analog of the Duhamel map of Proposition~\ref{prop:local}.

\smallskip

\noindent\emph{Polynomial bound on the nonlinearity.} By the algebra property of $H^{s+2}$ applied twice (note that $s+2>1/2$),
\[
  \|\rho(\Psi)\|_{H^{s+2}}\;\le\;\sum_l\mu_l\||\psi_l|^2\|_{H^{s+2}}\;\le\;C\sum_l\mu_l\|\psi_l\|_{H^{s+2}}^2\;\le\;C\,\mathcal{M}_0\,\|\Psi\|_{X_{T_*}}^2,
\]
and applying the algebra property once more on $V_{\rho(\Psi)}\psi_k=\rho(\Psi)\,\psi_k,$
\begin{equation}\label{eq:cubic_bound_finiteNLS}
  \|V_{\rho(\Psi)}\psi_k\|_{H^{s+2}}\;\le\;C\|\rho(\Psi)\|_{H^{s+2}}\|\psi_k\|_{H^{s+2}}\;\le\;C\,\mathcal{M}_0\,\|\Psi\|_{X_{T_*}}^3.
\end{equation}

\smallskip

\noindent\emph{Mapping a ball to itself.} By isometry of $e^{-ipt\Delta}$ on $H^{s+2}$ and the Duhamel inequality,
\[
  \|\Phi(\Psi)\|_{X_{T_*}}\;\le\;R_0\;+\;C|q|\,\mathcal{M}_0\,T_*\,\|\Psi\|_{X_{T_*}}^3.
\]
On the ball $B_{2R_0}:=\{\Psi:\|\Psi\|_{X_{T_*}}\le 2R_0\},$ the right-hand side can be bounded
\[
\Psi \in B_{2R_0} \implies \| \Phi(\Psi)\|_{X_{T_*}} \le R_0+8C|q|\mathcal{M}_0 T_*R_0^3,
\]
and, therefore, by choosing
\begin{equation}\label{eq:Tstar_bound_finiteNLS}
  T_*\;\le\;\frac{1}{8C|q|\,\mathcal{M}_0\,R_0^2}
\end{equation}
we have $\Phi(B_{2R_0}) \subseteq B_{2R_0}.$
\smallskip

\noindent\emph{Contraction.} For $\Psi^{(1)},\Psi^{(2)}\in B_{2R_0},$ polarize
\[
  V_{\rho^{(1)}}\psi_k^{(1)}-V_{\rho^{(2)}}\psi_k^{(2)}\;=\;V_{\rho^{(1)}-\rho^{(2)}}\,\psi_k^{(1)}\;+\;V_{\rho^{(2)}}(\psi_k^{(1)}-\psi_k^{(2)}),
\]
with $\rho^{(j)}:=\rho(\Psi^{(j)}).$ Using
\[
  |\psi_l^{(1)}|^2-|\psi_l^{(2)}|^2\;=\;(\psi_l^{(1)}-\psi_l^{(2)})\overline{\psi_l^{(1)}}\;+\;\psi_l^{(2)}\overline{(\psi_l^{(1)}-\psi_l^{(2)})}
\]
and then the algebra property, we get
\[
  \|\rho^{(1)}-\rho^{(2)}\|_{H^{s+2}}\;\le\;C\mathcal{M}_0\bigl(\|\Psi^{(1)}\|_{X_{T_*}}+\|\Psi^{(2)}\|_{X_{T_*}}\bigr)\,\|\Psi^{(1)}-\Psi^{(2)}\|_{X_{T_*}}.
\]
Substituting in the polarized expression and bounding each piece by the algebra property:
\begin{align*}
  \|V_{\rho^{(1)}-\rho^{(2)}}\,\psi_k^{(1)}\|_{H^{s+2}}
  &\;\le\;C\,\|\rho^{(1)}-\rho^{(2)}\|_{H^{s+2}}\,\|\psi_k^{(1)}\|_{H^{s+2}}\\
  &\;\le\;C\mathcal{M}_0\bigl(\|\Psi^{(1)}\|_{X_{T_*}}+\|\Psi^{(2)}\|_{X_{T_*}}\bigr)\,\|\Psi^{(1)}\|_{X_{T_*}}\,\|\Psi^{(1)}-\Psi^{(2)}\|_{X_{T_*}},\\[2pt]
  \|V_{\rho^{(2)}}\,(\psi_k^{(1)}-\psi_k^{(2)})\|_{H^{s+2}}
  &\;\le\;C\,\|\rho^{(2)}\|_{H^{s+2}}\,\|\psi_k^{(1)}-\psi_k^{(2)}\|_{H^{s+2}}\\
  &\;\le\;C\mathcal{M}_0\,\|\Psi^{(2)}\|_{X_{T_*}}^2\,\|\Psi^{(1)}-\Psi^{(2)}\|_{X_{T_*}}.
\end{align*}
Adding, taking $\max_k,$ and using $\|\Psi^{(j)}\|_{X_{T_*}}\le 2R_0,$
\[
\begin{aligned}
  \|V_{\rho^{(1)}}\Psi^{(1)}-V_{\rho^{(2)}}\Psi^{(2)}\|_{X_{T_*}}
  &\;\le\;C\mathcal{M}_0\bigl(\|\Psi^{(1)}\|_{X_{T_*}}+\|\Psi^{(2)}\|_{X_{T_*}}\bigr)^{\!2}\,\|\Psi^{(1)}-\Psi^{(2)}\|_{X_{T_*}}\\
  &\;\le\;16C\mathcal{M}_0 R_0^2\,\|\Psi^{(1)}-\Psi^{(2)}\|_{X_{T_*}},
\end{aligned}
\]
so the Duhamel inequality yields
\[
  \|\Phi(\Psi^{(1)})-\Phi(\Psi^{(2)})\|_{X_{T_*}}
  \;\le\;16C|q|\mathcal{M}_0 T_* R_0^2\,\|\Psi^{(1)}-\Psi^{(2)}\|_{X_{T_*}}.
\]
Sharpening~\eqref{eq:Tstar_bound_finiteNLS} to
\[
  T_*\;\le\;\frac{1}{32C|q|\,\mathcal{M}_0\,R_0^2}
\]
makes $\Phi$ a $\tfrac12$-contraction on $B_{2R_0},$ and Banach's fixed-point theorem produces a unique $\Psi\in B_{2R_0}\subset C([0,T_*];H^{s+2}(\T))^M$ with $\Phi(\Psi)=\Psi.$ 
That is,
\begin{equation}\label{eq:orbital_Duhamel}
  \psi_k(t)=e^{-ipt\Delta}\psi_k^0
  \;-\;iq\int_0^t e^{-ip(t-\tau)\Delta}\,V_{\rho(\tau)}\,\psi_k(\tau)\,d\tau,
  \qquad k=1,\dots,M.
\end{equation}
Differentiating~\eqref{eq:orbital_Duhamel} once in $t,$ it follows that
\[
  \psi_k\in C([0,T_*];H^{s+2}(\T))\cap C^1([0,T_*];H^s(\T)),
\]
and~\eqref{eq:Alber_system_finite_NLS} holds pointwise in $t$ as an equality in $H^s.$

\smallskip

\noindent\textbf{Step 2: Preservation of the entire orthonormal Gram matrix.} For scalar NLS, only $\|\psi\|_{L^2}$ is conserved. Here the entire Gram matrix is:
\begin{equation}\label{eq:gram_conservation}
  \langle\psi_k(t),\psi_l(t)\rangle\;=\;\delta_{kl}
  \qquad\forall\,k,l\in\{1,\dots,M\},\ \forall\,t.
\end{equation}
Differentiating $\langle\psi_k,\psi_l\rangle$ along the flow,
\[
  \partial_t\langle\psi_k,\psi_l\rangle
  =\langle-ip\Delta\psi_k-iqV_\rho\psi_k,\psi_l\rangle
  +\langle\psi_k,-ip\Delta\psi_l-iqV_\rho\psi_l\rangle
  =0.
\]
Note that we have established enough regularity for these operations to be well defined.

\smallskip

\noindent\textbf{Step 3: Operator-level packaging and conservation laws.} Set $\gamma(t):=\sum_k\mu_k|\psi_k(t)\rangle\langle\psi_k(t)|.$ Orthonormality~\eqref{eq:gram_conservation} and $\mu_k\ge 0$ make $\gamma(t)$ self-adjoint and non-negative with $\operatorname{rank}\le M.$ The kernel $K_\gamma(x,y,t)=\sum_k\mu_k\,\psi_k(x,t)\overline{\psi_k(y,t)}$ is smooth, so kernel-diagonal manipulations involving $\rho_\gamma(x,t)=K_\gamma(x,x,t)$ are direct.

\smallskip

\noindent\emph{Operator equation.} Differentiating $\gamma$ via the orbital ODEs and using orthonormality identifies
\[
  i\partial_t\gamma\;=\;p[\Delta,\gamma]\;+\;q[V_{\rho_\gamma},\gamma],
\]
i.e.\ $\gamma$ solves the operator Alber equation~\eqref{eq:Alber_op}.

\smallskip

\noindent\emph{Mass and Hilbert--Schmidt conservation.} Orthonormality of $\{\psi_k(t)\}_k$ for every $t$ (Step~2) gives
\[
  \|\gamma(t)\|_{\Sp^1}\;=\;\sum_{k=1}^M\mu_k\;=\;\mathcal{M},
  \qquad
  \|\gamma(t)\|_{\Sp^2}^{\,2}\;=\;\sum_{k=1}^M\mu_k^{\,2},
\]
the first two equalities of~\eqref{eq:finiteNLS_conservation}.

\smallskip

\noindent\emph{Energy conservation.} Define the orbital current
\[
  j(x,t)\;:=\;\sum_{k=1}^{M}\mu_k\,\mathrm{Im}\bigl(\overline{\psi_k(x,t)}\,\nabla\psi_k(x,t)\bigr).
\]
From~\eqref{eq:Alber_system_finite_NLS} and its complex conjugate, multiplied respectively by $\overline{\psi_k}$ and $\psi_k$ and subtracted (the $V_\rho$-term cancelling pointwise because $V_\rho$ is real),
\[
  \partial_t|\psi_k|^2\;=\;2p\,\nabla\!\cdot\!\mathrm{Im}\bigl(\overline{\psi_k}\nabla\psi_k\bigr).
\]
Multiplying by $\mu_k$ and summing yields the orbital continuity equation
\begin{equation}\label{eq:continuity_finiteNLS}
  \partial_t\rho\;-\;2p\,\nabla\!\cdot\!j\;=\;0.
\end{equation}
Multiplying~\eqref{eq:continuity_finiteNLS} by $\rho$ and integrating by parts on $\T,$
\begin{equation}\label{eq:rhoL2_dt_finiteNLS}
  \partial_t\|\rho(t)\|_{L^2}^{\,2}\;=\;2\!\int_\T\rho\,\partial_t\rho\,dx\;=\;4p\!\int_\T\rho\,\nabla\!\cdot\!j\,dx\;=\;-4p\!\int_\T\nabla\rho\cdot j\,dx.
\end{equation}
For the kinetic energy at the orbital level,
\[
  \mathcal{K}(t)\;:=\;\sum_k\mu_k\|\nabla\psi_k(t)\|_{L^2}^{\,2}\;=\;\tr(-\Delta\gamma(t)),
\]
differentiate and integrate by parts:
\[
\begin{aligned}
  \partial_t\mathcal{K}
  &\;=\;2\sum_k\mu_k\,\mathrm{Re}\langle\nabla\psi_k,\nabla\partial_t\psi_k\rangle_{L^2}
  \;=\;-2\sum_k\mu_k\,\mathrm{Re}\langle\Delta\psi_k,\partial_t\psi_k\rangle_{L^2}.
\end{aligned}
\]
Substituting $\partial_t\psi_k=-ip\Delta\psi_k-iqV_\rho\psi_k$ from~\eqref{eq:Alber_system_finite_NLS}, the $\Delta\psi_k$ contribution gives $-2\,\mathrm{Re}(-ip)\sum_k\mu_k\|\Delta\psi_k\|_{L^2}^{\,2}=0,$ and integration by parts on the $V_\rho\psi_k$ contribution (using that $V_\rho=\rho$ is real) yields
\[
  -2\sum_k\mu_k\,\mathrm{Re}\langle\Delta\psi_k,-iqV_\rho\psi_k\rangle_{L^2}
  \;=\;-2q\!\int_\T\nabla\rho\cdot j\,dx,
\]
so that
\begin{equation}\label{eq:Kcal_dt_finiteNLS}
  \partial_t\mathcal{K}(t)\;=\;-2q\!\int_\T\nabla\rho\cdot j\,dx.
\end{equation}
Combining~\eqref{eq:rhoL2_dt_finiteNLS} and~\eqref{eq:Kcal_dt_finiteNLS} with $E(\gamma)=p\,\tr(\Delta\gamma)+\tfrac{q}{2}\|\rho\|_{L^2}^{\,2}=-p\,\mathcal{K}+\tfrac{q}{2}\|\rho\|_{L^2}^{\,2},$
\[
  \partial_t E(\gamma(t))
  \;=\;2pq\!\int_\T\nabla\rho\cdot j\,dx\;-\;2pq\!\int_\T\nabla\rho\cdot j\,dx
  \;=\;0,
\]
proving $E(\gamma(t))=E(\gamma(0))=:E_0,$ the third equality of~\eqref{eq:finiteNLS_conservation}.

\smallskip

\noindent\textbf{Step 4: A priori $H^1$ bound (\emph{inspired by scalar NLS, but one more step is needed}).} The difference  is due to the position density having contributions from all coordinates, $\rho=\sum\mu_k|\psi_k|^2.$ To proceed we reformulate
 $\|\rho\|_{L^2}^2=\|\sqrt\rho\|_{L^4}^4$ and apply the $1$D Gagliardo--Nirenberg inequality (Lemma~\ref{lem:GN1d}) to $u=\sqrt\rho:$
  \begin{equation}\label{eq:GN_finiteNLS}
  \begin{aligned}
    \|\sqrt\rho\|_{L^4}^4
    &\;\le\;\frac{1}{|\T|}\|\sqrt\rho\|_{L^2}^4+2\|\sqrt\rho\|_{L^2}^3\|\nabla\sqrt\rho\|_{L^2}\\
    &\;=\;\frac{\mathcal{M}^2}{|\T|}+2\mathcal{M}^{3/2}\|\nabla\sqrt\rho\|_{L^2}.
  \end{aligned}
  \end{equation}

 To control $\|\nabla\sqrt\rho\|_{L^2}^2,$ observe first of all that
  \begin{gather*}
    |\nabla\rho|
    =\Bigl|2\sum_k\mu_k\,\mathrm{Re}\bigl(\overline{\psi_k}\,\nabla\psi_k\bigr)\Bigr|
    \;\le\; 2\sum_k\bigl(\sqrt{\mu_k}\,|\psi_k|\bigr)\bigl(\sqrt{\mu_k}\,|\nabla\psi_k|\bigr)\\
    \le\; 2\sqrt{\rho}\,\Bigl(\sum_k\mu_k|\nabla\psi_k|^2\Bigr)^{1/2}
    \quad\Longrightarrow\quad
    |\nabla\rho|^2\;\le\;4\rho\sum_k\mu_k|\nabla\psi_k|^2,
  \end{gather*}
  and thus
\[
\begin{aligned}
  \|\nabla\sqrt{\rho+\varepsilon}\|_{L^2}^2
  &\;=\;\int_\T\frac{|\nabla\rho|^2}{4(\rho+\varepsilon)}\,dx\\
  &\;\le\;\int_\T\frac{\rho}{\rho+\varepsilon}\sum_k\mu_k|\nabla\psi_k|^2\,dx
   \;\le\;\sum_k\mu_k\|\nabla\psi_k\|_{L^2}^2.
\end{aligned}
\]
Now  Lemma~\ref{lem:LSC} yields
  \begin{equation}\label{eq:HO_finiteNLS}
    \|\nabla\sqrt\rho\|_{L^2}^2=\lim\limits_{\varepsilon\to 0^+}     \|\nabla\sqrt{\rho+\varepsilon}\|_{L^2}^2\;\le\;\sum_k\mu_k\|\nabla\psi_k\|_{L^2}^2,
  \end{equation}
  namely the Hoffmann-Ostenhof bound for a finite-rank system.

\smallskip

Now, to collect everything: with the kinetic energy $\mathcal{K}(t)=\tr(-\Delta\gamma(t))=\sum_k\mu_k\|\nabla\psi_k(t)\|_{L^2}^2$ as in Step~3, the bounds~\eqref{eq:GN_finiteNLS} and~\eqref{eq:HO_finiteNLS} chain into
\begin{equation}\label{eq:rhoL2_chain_finiteNLS}
  \|\rho(t)\|_{L^2}^2
  \;\le\;\frac{\mathcal{M}^2}{|\T|}+2\mathcal{M}^{3/2}\|\nabla\sqrt{\rho(t)}\|_{L^2}
  \;\le\;\frac{\mathcal{M}^2}{|\T|}+2\mathcal{M}^{3/2}\sqrt{\mathcal{K}(t)}.
\end{equation}

Energy conservation, with $\tr(\Delta\gamma)=-\mathcal{K},$ rearranges to the energy identity
\begin{equation}\label{eq:energy_identity_finiteNLS}
  \mathcal{K}(t)\;=\;\mathcal{K}(0)+\frac{q}{2p}\bigl(\|\rho(t)\|_{L^2}^2-\|\rho_0\|_{L^2}^2\bigr).
\end{equation}

\emph{Focusing case ($pq>0$).} Here $q/(2p)=|q|/(2|p|)>0,$ so dropping the non-positive $-\|\rho_0\|_{L^2}^2$ term in~\eqref{eq:energy_identity_finiteNLS} and substituting~\eqref{eq:rhoL2_chain_finiteNLS},
\[
  \mathcal{K}(t)
  \;\le\;\mathcal{K}(0)+\frac{|q|}{2|p|}\Bigl(\frac{\mathcal{M}^2}{|\T|}+2\mathcal{M}^{3/2}\sqrt{\mathcal{K}(t)}\Bigr)
  \;=\; B+A\sqrt{\mathcal{K}(t)},
\]
with
\[
  A\;:=\;\frac{|q|\,\mathcal{M}^{3/2}}{|p|},
  \qquad
  B\;:=\;\mathcal{K}(0)+\frac{|q|\,\mathcal{M}^2}{2|p|\,|\T|}.
\]
Setting $x:=\sqrt{\mathcal{K}(t)}\ge 0,$ this is the quadratic inequality $x^2-Ax-B\le 0,$ whose non-negative root is $x_\star=\bigl(A+\sqrt{A^2+4B}\bigr)/2.$ Hence
\begin{equation}\label{eq:apriori_finiteNLS_foc}
  \mathcal{K}(t)\;\le\;x_\star^2
  \qquad\text{for all }t\text{ in the lifespan.}
\end{equation}

\emph{Defocusing case ($pq<0$).} Here $q/(2p)=-|q|/(2|p|)<0,$ so dropping the non-positive $-\|\rho(t)\|_{L^2}^2$ term in~\eqref{eq:energy_identity_finiteNLS} gives directly
\begin{equation}\label{eq:apriori_finiteNLS_def}
  \mathcal{K}(t)\;\le\;\mathcal{K}(0)+\frac{|q|}{2|p|}\,\|\rho_0\|_{L^2}^2,
\end{equation}
controlled by initial data without invoking~\eqref{eq:rhoL2_chain_finiteNLS}.

In both cases, the kinetic energy $\mathcal{K}(t)$ is bounded a priori on the lifespan by a finite constant depending only on the size of the initial data and on $p,q,|\T|.$ Combined with conservation of mass $\mathcal{M},$ this yields a uniform $H^1\Sp^1$ bound at the orbital level: there exists $\bar Y=\bar Y(\|\gamma_0\|_{H^1\Sp^1},p,q,|\T|)<\infty$ such that
\begin{equation}\label{eq:Ybar_finiteNLS}
  \mathcal{M}+\mathcal{K}(t)\;\le\;\bar Y\qquad\text{for all }t\text{ in the lifespan.}
\end{equation}

The new ingredient compared to scalar pure-state NLS is the chain~\eqref{eq:rhoL2_chain_finiteNLS}: Hoffmann--Ostenhof~\eqref{eq:HO_finiteNLS} bridges $\|\rho\|_{L^2}^2$ (potential energy term) and $\sum_k\mu_k\|\nabla\psi_k\|_{L^2}^2$ (kinetic energy), and the GN constant on the right of~\eqref{eq:GN_finiteNLS} is controlled by $\mathcal{M}=\|\sqrt\rho\|_{L^2}^2,$ which is preserved precisely because of the Gram-matrix conservation of Step~2.

\smallskip

\noindent\textbf{Step 5: Global extension for the $H^{s+2}$ norm.} Now we will show that the a priori $H^1$ bound precludes blow up of the $H^{s+2}$ norm. Set
\[
  \Sigma_s(t)\;:=\;\sum_k\mu_k\|\psi_k(t)\|_{H^{s+2}}^2.
\]
Throughout this step, $C_s>0$ denotes a generic constant collecting the Kato--Ponce constant of~\eqref{eq:KP_commutator} at exponent $s+2,$ the algebra-of-products constant on $H^{s+2}(\T),$ and the Sobolev embedding constant $H^1(\T)\hookrightarrow L^\infty(\T)$ in $d=1.$ All three depend only on $s$ and $|\T|;$ in particular $C_s$ is independent of the orbital data, the weights $\mu_k,$ and the rank $M$ of the system.

\smallskip

\noindent\emph{Reduction to a commutator term.} Differentiating $\Sigma_s$ in time,
\[
  \partial_t\Sigma_s
  \;=\;2\sum_k\mu_k\,\mathrm{Re}\,\langle\partial_t\psi_k,\psi_k\rangle_{H^{s+2}},
\]
and substituting the time derivative from~\eqref{eq:Alber_system_finite_NLS} 
\begin{equation}\label{eq:Sigma_dt_split}
  \partial_t\Sigma_s
  \;=\;2\sum_k\mu_k\,\mathrm{Im}\bigl[p\,\langle\Delta\psi_k,\psi_k\rangle_{H^{s+2}}\;+\;q\,\langle V_\rho\psi_k,\psi_k\rangle_{H^{s+2}}\bigr].
\end{equation}
The real terms drop out, leaving only the commutator:
\[
  \partial_t\Sigma_s
  \;=\;2q\sum_k\mu_k\,\mathrm{Im}\,\langle[\langle D\rangle^{s+2},V_\rho]\psi_k,\langle D\rangle^{s+2}\psi_k\rangle_{L^2}.
\]

By Cauchy--Schwarz in the $L^2$ pairing,
\begin{equation}\label{eq:Sigma_dt_KP}
  |\partial_t\Sigma_s|
  \;\le\;2|q|\sum_k\mu_k\,\|[\langle D\rangle^{s+2},V_\rho]\psi_k\|_{L^2}\,\|\psi_k\|_{H^{s+2}}.
\end{equation}

\smallskip

\noindent\emph{Kato--Ponce on the commutator.} Inequality~\eqref{eq:KP_commutator} applied with $f=\rho,$ $g=\psi_k,$ and $s$ replaced by $s+2$ gives
\begin{equation}\label{eq:KP_comm_finiteNLS}
  \|[\langle D\rangle^{s+2},V_\rho]\psi_k\|_{L^2}
  \;\le\;C\bigl(\|\rho\|_{H^{s+2}}\|\psi_k\|_{L^\infty}+\|\rho\|_{L^\infty}\|\psi_k\|_{H^{s+2}}\bigr).
\end{equation}

\smallskip

\noindent\emph{Bounds on $\|\rho\|_{H^{s+2}}$ and $\|\rho\|_{L^\infty}.$} For the $H^{s+2}$ norm, using Lemma~\ref{lem:KP} on $|\psi_k|^2=\psi_k\overline{\psi_k}$ and Sobolev embedding $H^1\hookrightarrow L^\infty$ in $d=1,$
\[
\begin{aligned}
  \|\rho\|_{H^{s+2}}
  &\;\le\;\sum_k\mu_k\||\psi_k|^2\|_{H^{s+2}}
   \;\le\;C_s\sum_k\mu_k\|\psi_k\|_{H^{s+2}}\|\psi_k\|_{L^\infty}\\
  &\;\le\;C_s\sum_k\mu_k\|\psi_k\|_{H^{s+2}}\|\psi_k\|_{H^1}.
\end{aligned}
\]
Cauchy--Schwarz on the sum with weights $\sqrt{\mu_k}\|\psi_k\|_{H^{s+2}}$ and $\sqrt{\mu_k}\|\psi_k\|_{H^1}$ gives
\begin{equation}\label{eq:rho_Hs2_finiteNLS}
  \|\rho\|_{H^{s+2}}
  \;\le\;C_s\sqrt{\Sigma_s(t)}\,\sqrt{\sum_k\mu_k\|\psi_k\|_{H^1}^2}
  \;=\;C_s\sqrt{\Sigma_s(t)}\,\sqrt{\mathcal{M}+\mathcal{K}},
\end{equation}
where in the last step we used $\sum_k\mu_k\|\psi_k\|_{H^1}^2=\sum_k\mu_k(\|\psi_k\|_{L^2}^2+\|\nabla\psi_k\|_{L^2}^2)=\mathcal{M}+\mathcal{K}$ by orthonormality of the $\psi_k.$ For $\|\rho\|_{L^\infty},$ the same Sobolev embedding gives
\begin{equation}\label{eq:rho_Linfty_finiteNLS}
  \|\rho\|_{L^\infty}\;\le\;\sum_k\mu_k\|\psi_k\|_{L^\infty}^2\;\le\;C_s\sum_k\mu_k\|\psi_k\|_{H^1}^2\;=\;C_s(\mathcal{M}+\mathcal{K}).
\end{equation}

\smallskip

\noindent\emph{Closing Gr\"onwall.} Substituting~\eqref{eq:KP_comm_finiteNLS} into~\eqref{eq:Sigma_dt_KP}, the right-hand side splits as
\[
  |\partial_t\Sigma_s|
  \;\le\;2|q|C_s\,\|\rho\|_{H^{s+2}}\sum_k\mu_k\|\psi_k\|_{L^\infty}\|\psi_k\|_{H^{s+2}}
  \;+\;2|q|C_s\,\|\rho\|_{L^\infty}\,\Sigma_s.
\]
The first term is bounded using Cauchy--Schwarz in $k$
\[
  \sum_k\mu_k\|\psi_k\|_{L^\infty}\|\psi_k\|_{H^{s+2}}
  \;\le\;C_s\sqrt{\Sigma_s}\,\sqrt{\mathcal{M}+\mathcal{K}},
\]
and combining with~\eqref{eq:rho_Hs2_finiteNLS}
\[
  \|\rho\|_{H^{s+2}}\sum_k\mu_k\|\psi_k\|_{L^\infty}\|\psi_k\|_{H^{s+2}}
  \;\le\;C_s\bigl(\sqrt{\Sigma_s}\,\sqrt{\mathcal{M}+\mathcal{K}}\bigr)^{\!2}
  \;=\;C_s(\mathcal{M}+\mathcal{K})\,\Sigma_s.
\]
Together with~\eqref{eq:rho_Linfty_finiteNLS},
\[
  |\partial_t\Sigma_s|\;\le\;C_s\,|q|\,(\mathcal{M}+\mathcal{K}(t))\,\Sigma_s(t),
\]
with $C_s$ uniform in everything but $s$ and $|\T|.$ By the a priori bound~\eqref{eq:Ybar_finiteNLS} of Step~4, $\mathcal{M}+\mathcal{K}(t)\le \bar Y$ uniformly on the lifespan. Gr\"onwall yields
\begin{equation}\label{eq:Sigma_growth}
  \Sigma_s(t)\;\le\;\Sigma_s(0)\,e^{C_s\,|q|\,\bar Y\,|t|},
\end{equation}
ruling out finite-time blow-up of $\Sigma_s$ and hence extending the local $H^{s+2}$ solution from Step~1 to all $t\in\R.$ The spectral identity~\eqref{eq:spectral_identity} translates~\eqref{eq:Sigma_growth} into $\gamma\in C(\R;H^{s+2}\Sp^1(\T))$ with the same exponential-in-$t$ bound.

\smallskip

\noindent\textbf{Recap of the differences from scalar pure-state NLS.} Steps~1 and~5 are componentwise repackaging of standard NLS local theory at $H^{s+2}$ with the $C^1H^s$ time regularity coming from differentiating Duhamel; the new ingredients are Step~2 (the entire orthonormal Gram matrix is preserved, not just $\|\psi\|_{L^2},$ thanks to the shared real potential $V_\rho$) and Step~4 (the focusing $H^1$ a priori bound is closed via Hoffmann--Ostenhof + Gagliardo-Nirenberg, as opposed to direct Gagliardo-Nirenberg applied directly to $\|\psi\|_{L^4}^4$ in the scalar case. Step~3 is the operator-level packaging: $\gamma=\sum\mu_k|\psi_k\rangle\langle\psi_k|$ inherits self-adjointness and non-negativity automatically from the orbital construction.
\end{proof}

\begin{lemma}[Qualitative properties of operator-valued Alber solutions]
\label{lem:finite_system}
Let $T>0$ and let $\gamma\in C([0,T];H^1\Sp^1(\T))$ be a mild solution of~\eqref{eq:Alber_op} with self-adjoint, non-negative datum $\gamma_0=\gamma_0^*\ge 0.$ Then:
\begin{enumerate}
\item[(a)] \emph{Self-adjointness and non-negativity.}
\[
  \gamma(t)=\gamma(t)^{*}\ge 0\qquad\text{for every }t\in[0,T].
\]
\item[(b)] \emph{Conservation laws and uniform $H^1\Sp^1$ bound.} The mass, Hilbert--Schmidt norm, and energy~\eqref{eq:E_def} are conserved:
\begin{equation}\label{eq:lfs_conservation}
\begin{aligned}
  &\|\gamma(t)\|_{\Sp^1}\;=\;\|\gamma_0\|_{\Sp^1}\;=:\;\mathcal{M},
  \qquad
  \|\gamma(t)\|_{\Sp^2}\;=\;\|\gamma_0\|_{\Sp^2},\\
  &E(\gamma(t))\;=\;E(\gamma_0)\;=:\;E_0.
\end{aligned}
\end{equation}
Moreover, 
\begin{equation}\label{eq:lfs_apriori}
  \|\gamma(t)\|_{H^1\Sp^1}\;=\;\mathcal{M}+\mathcal{K}(t)\;\le\;\bar Y\qquad\text{for every }t\in[0,T]
\end{equation}
with $\bar Y=\bar Y(\|\gamma_0\|_{H^1\Sp^1},p,q,|\T|)$ given by Lemma~\ref{lem:finite_NLS}.
\item[(c)] \emph{Propagation of higher Sobolev regularity.} If $\gamma_0\in H^{s+2}\Sp^1(\T)$ for some $s\ge 0,$ then
\[
  \gamma\in C([0,T];H^{s+2}\Sp^1(\T))\cap C^1([0,T];H^s\Sp^1(\T)),
\]
the Alber equation~\eqref{eq:Alber_op} holds pointwise in $t$ as an identity in $H^s\Sp^1,$ and
\begin{equation}\label{eq:lfs_Hs_bound}
  \|\gamma(t)\|_{H^{s+2}\Sp^1}\;\le\;\|\gamma_0\|_{H^{s+2}\Sp^1}\,e^{C_s\,|q|\,\bar Y\,|t|}
  \qquad\text{for every }t\in[0,T],
\end{equation}
where $C_s$ is the constant appearing in the orbital-level Gr\"onwall closure of Lemma~\ref{lem:finite_NLS} (eq.~\eqref{eq:Sigma_growth}).
\end{enumerate}
\end{lemma}

\noindent Note that the timescale $T$ here is arbitrary; in particular, unlike Proposition \ref{prop:local}, it is not connected with the size of $\gamma_0.$

\begin{proof}
\noindent\textbf{Setup: finite-rank approximants.} For $N\in\mathbb{N},$ define the Galerkin truncation
\[
  \gamma_0^{(N)} \;:=\; P_N\,\gamma_0\,P_N.
\]
We collect the properties of $\gamma_0^{(N)}$ that will be used in what follows. By Proposition~\ref{prop:galerkin}(i), $\gamma_0^{(N)}$ is finite-rank with $\operatorname{rank}\gamma_0^{(N)}\le 2N+1,$ and lies in $H^\sigma\Sp^1(\T)$ for every $\sigma\ge 0$ --- in particular, it is smooth at the operator level. Self-adjointness and non-negativity are inherited from the datum: since $\gamma_0=\gamma_0^*\ge 0,$ Proposition~\ref{prop:galerkin}(ii) gives $\gamma_0^{(N)}=(\gamma_0^{(N)})^*\ge 0.$ The $H^1\Sp^1$ norm is dominated by that of the datum uniformly in $N,$
\[
  \|\gamma_0^{(N)}\|_{H^1\Sp^1}\;\le\;\|\gamma_0\|_{H^1\Sp^1}
\]
(Proposition~\ref{prop:galerkin}(iii)), and $\gamma_0^{(N)}\to\gamma_0$ in $H^1\Sp^1$ as $N\to\infty$ --- and, when $\gamma_0\in H^{s+2}\Sp^1,$ also in $H^{s+2}\Sp^1$ --- by Proposition~\ref{prop:galerkin}(iv). Finally, since $\gamma_0^{(N)}$ is self-adjoint and non-negative with finite-dimensional range, it admits a finite spectral decomposition
\[
  \gamma_0^{(N)} \;=\; \sum_{k=1}^{M_N}\mu_k^{(N)}\,|\psi_k^{(N)}\rangle\langle\psi_k^{(N)}|,
  \qquad M_N\le 2N+1,
\]
with weights $\mu_k^{(N)}\ge 0$ and orthonormal orbitals $\psi_k^{(N)}\in P_N L^2(\T)\subset C^\infty(\T)$ (Proposition~\ref{prop:galerkin}(vi)).

Thus, Lemma~\ref{lem:finite_NLS} applies and produces a unique global solution
\[
  \gamma^{(N)}\in C(\R;\,H^\sigma\Sp^1(\T)) \quad \text{for every } \sigma\ge 0,
\]
self-adjoint and non-negative, with $\operatorname{rank}\gamma^{(N)}(t)\le 2N+1,$ solving the operator Alber equation~\eqref{eq:Alber_op}, and satisfying the conservation laws~\eqref{eq:finiteNLS_conservation} together with the uniform $H^1\Sp^1$ bound~\eqref{eq:Ybar_finiteNLS_stmt}. All three clauses of the present lemma are obtained by passing the corresponding properties of $\gamma^{(N)}$ to the limit $N\to\infty.$

\smallskip
\noindent\textbf{Step 1: Convergence $\gamma^{(N)}\to\gamma$ in $C([0,T];H^1\Sp^1).$} We will show that, given any $\varepsilon>0,$ we can always choose a $N_0(\varepsilon)\in\mathbb{N}$ so that
\begin{equation}\label{eq:imdef3}
N\ge N_0 \implies \|\gamma^{(N)}-\gamma\|_{C([0,T];H^1\Sp^1)}\le \varepsilon.
\end{equation}

Both $\gamma$ and $\gamma^{(N)}$ satisfy the Duhamel equation~\eqref{eq:Duhamel}, and we timestep using the $H^1\Sp^1$-level Lipschitz stability of Proposition~\ref{prop:local}. Set 
\[
\Lambda:=\sup_{t\in[0,T]}\|\gamma(t)\|_{H^1\Sp^1}<\infty
\]
 by continuity of $\gamma$ on $[0,T].$ Set $\tau:=1/(8C(\Lambda+1)),$ so that $\delta(\tau)\ge\Lambda+1$ (Remark~\ref{rem:delta_T}). Partition $[0,T]$ into $K=\lceil T/\tau\rceil$ sub-intervals $0=t_0<t_1<\cdots<t_K=T$ with $t_{k+1}-t_k\le\tau.$ By Proposition~\ref{prop:galerkin}(iv), we can pick $N_0$ so that 
 \begin{equation}\label{eq:n0eps}
 2^K\|\gamma_0^{(N)}-\gamma_0\|_{H^1\Sp^1}\le \min(\varepsilon,1) \quad\mbox{for} \quad N\ge N_0.
  \end{equation}

Fix $N\ge N_0.$ We claim by induction on $k$ that
\begin{equation}\label{eq:lfs_ind}
  \|\gamma^{(N)}(t_k)-\gamma(t_k)\|_{H^1\Sp^1}\le 2^k\|\gamma_0^{(N)}-\gamma_0\|_{H^1\Sp^1}
  \qquad(0\le k\le K).
\end{equation}
The base case is a trivial equality. For the inductive step, we first use the inductive hypothesis to get
\[
\begin{aligned}
  \|\gamma^{(N)}(t_k)\|_{H^1\Sp^1}
  &\le \|\gamma^{(N)}(t_k)-\gamma(t_k)\|_{H^1\Sp^1} + \|\gamma(t_k)\|_{H^1\Sp^1}\\
  &\le 2^k\|\gamma_0^{(N)}-\gamma_0\|_{H^1\Sp^1} + \Lambda\\
  &\le 1+ \Lambda\;\le\;\delta(\tau);
\end{aligned}
\]
where we also used the assumption \eqref{eq:n0eps}.
Thus,  both $\gamma$ and $\gamma^{(N)}$ on $[t_k,t_{k+1}]$ have data of $H^1\Sp^1$ size no more than $\delta(\tau)$ and the Lipschitz stability~\eqref{eq:lwp_stab} of Proposition~\ref{prop:local} applies. Since the $X_\tau$ norm is the sup over $[t_k,t_{k+1}],$ this gives
\[
  \sup_{t\in[t_k,t_{k+1}]}\|\gamma^{(N)}(t)-\gamma(t)\|_{H^1\Sp^1}\;\le\; 2\,\|\gamma^{(N)}(t_k)-\gamma(t_k)\|_{H^1\Sp^1};
\]
evaluating at $t=t_{k+1}$ closes the induction. 

Combining with \eqref{eq:n0eps}, \eqref{eq:imdef3} follows. The embeddings $H^1\Sp^1\hookrightarrow\Sp^1\hookrightarrow\Sp^2\hookrightarrow\Sp^\infty$ then yield the convergences in all Schatten classes. 

\smallskip
\noindent\textbf{Step 2: Clause (a).} By Lemma~\ref{lem:finite_NLS}, $\gamma^{(N)}(t)=\gamma^{(N)}(t)^*\ge 0$ for every $t.$ By Step~1, $\gamma^{(N)}\to\gamma$ in $C([0,T];\Sp^\infty),$ and both self-adjointness and non-negativity are closed under $\Sp^\infty$ convergence. Hence $\gamma(t)=\gamma(t)^*\ge 0$ on $[0,T].$

\smallskip
\noindent\textbf{Step 3: Clause (b).}

\smallskip
\noindent\emph{Conservation laws at the approximant level.} Lemma~\ref{lem:finite_NLS}, eq.~\eqref{eq:finiteNLS_conservation}, gives
\begin{gather*}
  \|\gamma^{(N)}(t)\|_{\Sp^1}\;=\;\|\gamma_0^{(N)}\|_{\Sp^1},\qquad
  \|\gamma^{(N)}(t)\|_{\Sp^2}\;=\;\|\gamma_0^{(N)}\|_{\Sp^2},\\
  E\bigl(\gamma^{(N)}(t)\bigr)\;=\;E\bigl(\gamma_0^{(N)}\bigr).
\end{gather*}

\smallskip
\noindent\emph{Passage to the limit.} By Step~1, $\gamma^{(N)}\to\gamma$ in $C([0,T];H^1\Sp^1),$ hence in $C([0,T];\Sp^1)$ and $C([0,T];\Sp^2).$ Proposition~\ref{prop:galerkin}(iii)--(iv) gives $\gamma_0^{(N)}\to\gamma_0$ in $H^1\Sp^1,$ hence in $\Sp^1$ and $\Sp^2.$ Passing to the limit in the first two identities gives $\|\gamma(t)\|_{\Sp^1}=\|\gamma_0\|_{\Sp^1}=:\mathcal{M}$ and $\|\gamma(t)\|_{\Sp^2}=\|\gamma_0\|_{\Sp^2}.$ For the energy, $\gamma\mapsto E(\gamma)$ is continuous on $H^1\Sp^1$: $\gamma\mapsto\tr(-\Delta\gamma)$ is bounded linear $H^1\Sp^1\to\R$ via~\eqref{eq:H1S1_split}, and $\gamma\mapsto\rho_\gamma$ is bounded linear $H^1\Sp^1\to L^2$ by Lemma~\ref{lem:trace} at $s=1$ (cf.\ Lemma~\ref{lem:E_finite}); hence $E(\gamma^{(N)}(t))\to E(\gamma(t))$ and $E(\gamma_0^{(N)})\to E(\gamma_0),$ giving $E(\gamma(t))=E(\gamma_0)=:E_0.$ This proves~\eqref{eq:lfs_conservation}.

\smallskip
\noindent\emph{A priori $H^1\Sp^1$ bound.} By Lemma~\ref{lem:finite_NLS}, eq.~\eqref{eq:Ybar_finiteNLS_stmt} applied to $(\gamma_0^{(N)},p,q,|\T|),$
\[
\begin{aligned}
  \mathcal{M}^{(N)}+\mathcal{K}^{(N)}(t)
  &\;\le\;\bar Y(\|\gamma_0^{(N)}\|_{H^1\Sp^1},p,q,|\T|)\\
  &\;\le\;\bar Y(\|\gamma_0\|_{H^1\Sp^1},p,q,|\T|)\;=\;\bar Y
  \qquad\text{for every }t\in[0,T],
\end{aligned}
\]
where $\mathcal{M}^{(N)}:=\|\gamma_0^{(N)}\|_{\Sp^1}$ and $\mathcal{K}^{(N)}(t):=\tr(-\Delta\gamma^{(N)}(t)),$ and the second inequality uses monotonicity of $\bar Y$ in its first argument together with $\|\gamma_0^{(N)}\|_{H^1\Sp^1}\le\|\gamma_0\|_{H^1\Sp^1}$ (Proposition~\ref{prop:galerkin}(iii)). The bound is uniform in $N.$ Step~1 gives $\gamma^{(N)}(t)\to\gamma(t)$ uniformly in $t$ in $H^1\Sp^1,$ so $\mathcal{M}^{(N)}\to\mathcal{M}$ and $\mathcal{K}^{(N)}(t)\to\mathcal{K}(t)$ uniformly in $t.$ Passing to the limit yields $\mathcal{M}+\mathcal{K}(t)\le\bar Y;$ and $\mathcal{M}+\mathcal{K}(t)=\|\gamma(t)\|_{H^1\Sp^1}$ by~\eqref{eq:H1S1_split} together with non-negativity from (a).

\smallskip
\noindent\textbf{Step 4: Clause (c).} Suppose $\gamma_0\in H^{s+2}\Sp^1$ for some $s\ge 0.$

\smallskip
\noindent\emph{Uniform bound on $\{\gamma^{(N)}\}.$} By Proposition~\ref{prop:galerkin}(iii)--(iv), $\gamma_0^{(N)}\to\gamma_0$ in $H^{s+2}\Sp^1$ with $\|\gamma_0^{(N)}\|_{H^{s+2}\Sp^1}\le\|\gamma_0\|_{H^{s+2}\Sp^1}.$ Lemma~\ref{lem:finite_NLS}, Step~5, eq.~\eqref{eq:Sigma_growth} translated by the spectral identity, together with the monotone majorant $\bar Y(\|\gamma_0^{(N)}\|_{H^1\Sp^1})\le\bar Y$ used in Step~3, gives
\begin{equation}\label{eq:lfs_Hs_uniform}
  \|\gamma^{(N)}(t)\|_{H^{s+2}\Sp^1}\;\le\;\|\gamma_0^{(N)}\|_{H^{s+2}\Sp^1}\,e^{C_s|q|\,\bar Y\,|t|}\;\le\;\|\gamma_0\|_{H^{s+2}\Sp^1}\,e^{C_s|q|\,\bar Y\,|t|},
\end{equation}
uniform in $N.$ In particular $\sup_{N\ge N_0,\,t\in[0,T]}\|\gamma^{(N)}(t)\|_{H^{s+2}\Sp^1}<\infty.$

\smallskip
\noindent\emph{$L^\infty_t H^{s+2}\Sp^1$ bound on $\gamma.$} Step~1 gives $\gamma^{(N)}(t)\to\gamma(t)$ in $\Sp^\infty$ uniformly in $t.$ The map $A\mapsto\|A\|_{H^{s+2}\Sp^1}$ is lower semicontinuous under $\Sp^\infty$ convergence: this is the Schatten-class analog of Lemma~\ref{lem:LSC}, obtained by writing $\|A\|_{H^{s+2}\Sp^1}=\|\langle D\rangle^{s+2}A\langle D\rangle^{s+2}\|_{\Sp^1}$ and using the lower semicontinuity of the $\Sp^1$ norm under Banach--Alaoglu weak-$*$ convergence (cf.\ \cite[Thm.~2.20]{Simon2005}). Applied to~\eqref{eq:lfs_Hs_uniform},
\begin{equation}\label{eq:lfs_Hs_pointwise}
  \|\gamma(t)\|_{H^{s+2}\Sp^1}\;\le\;\liminf_{N\to\infty}\|\gamma^{(N)}(t)\|_{H^{s+2}\Sp^1}\;\le\;\|\gamma_0\|_{H^{s+2}\Sp^1}\,e^{C_s|q|\,\bar Y\,|t|},
\end{equation}
which is~\eqref{eq:lfs_Hs_bound}.

\smallskip
\noindent\emph{Continuity $\gamma\in C([0,T];H^{s+2}\Sp^1).$} $\gamma$ satisfies the operator Duhamel equation~\eqref{eq:Duhamel}. The free flow $S(t)\gamma_0$ is continuous in $t$ in $H^{s+2}\Sp^1$ (the free flow is an isometric group on every Schatten--Sobolev space). The bilinear bound (Proposition~\ref{prop:bilinear} at $\sigma=s+2$) and~\eqref{eq:lfs_Hs_pointwise} give
\[
  \sup_{\tau\in[0,T]}\|[V_{\rho_\gamma(\tau)},\gamma(\tau)]\|_{H^{s+2}\Sp^1}\;\le\;C\sup_{\tau\in[0,T]}\|\gamma(\tau)\|_{H^{s+2}\Sp^1}^2\;<\;\infty,
\]
so $\int_0^t S(t-\tau)[V_{\rho_\gamma(\tau)},\gamma(\tau)]\,d\tau$ is continuous in $t$ as an $H^{s+2}\Sp^1$-valued Bochner integral. Hence $\gamma\in C([0,T];H^{s+2}\Sp^1).$

\smallskip
\noindent\emph{$C^1$ part.} The Alber equation~\eqref{eq:Alber_op} reads $i\partial_t\gamma=p[\Delta,\gamma]+q[V_{\rho_\gamma},\gamma].$ The kinetic term lies in $H^s\Sp^1$ via the boundedness of $\Delta:H^{s+2}\Sp^1\to H^s\Sp^1,$ and the potential term lies in $H^{s+2}\Sp^1\hookrightarrow H^s\Sp^1$ by Proposition~\ref{prop:bilinear} at $\sigma=s+2.$ Both are continuous in $t$ in $H^s\Sp^1$ (using continuity of $\gamma$ in $H^{s+2}\Sp^1$ and continuity of the bilinear nonlinearity), so $\partial_t\gamma\in C([0,T];H^s\Sp^1),$ i.e.\ $\gamma\in C^1([0,T];H^s\Sp^1).$
\end{proof}



\section{Proof of Theorem~\ref{thm:main}}\label{sec:proof_thms}

The argument is sign-agnostic: $pq>0$ vs $pq<0$ enters only through the closed form of $\bar Y$ in Remark~\ref{rem:barY}. We combine the local well-posedness of Proposition~\ref{prop:local} with the uniform a priori $H^1\Sp^1$ bound of Lemma~\ref{lem:finite_system}(b) to extend the solution for all time. Moreover, we obtain clause (b) of Theorem~\ref{thm:main} from Lemma~\ref{lem:finite_system}(c).

\medskip

\noindent Let $\bar Y=\bar Y(\|\gamma_0\|_{H^1\Sp^1},p,q,|\T|)$ be the uniform a priori $H^1\Sp^1$ bound of Lemma~\ref{lem:finite_system}(b), reported in Remark~\ref{rem:barY}. Set $T:=T(\bar Y)=1/(8C\bar Y)$ (Remark~\ref{rem:delta_T}).

\medskip

\noindent\textbf{Step 1: Solve on $[0,T]$.} Since $\|\gamma_0\|_{H^1\Sp^1}\le\bar Y$ and $T=T(\bar Y),$ Proposition~\ref{prop:local} yields a unique $\gamma\in C([0,T];H^1\Sp^1)$ solving~\eqref{eq:Alber_op} on $[0,T].$

\medskip

\noindent\textbf{Step 2: Apply the a priori bound.} Lemma~\ref{lem:finite_system}(b) gives $\|\gamma(t)\|_{H^1\Sp^1}\le\bar Y$ for every $t\in[0,T].$

\medskip

\noindent\textbf{Step 3: Restart and concatenate.} Since $\|\gamma(T)\|_{H^1\Sp^1}\le\bar Y$ and (by Lemma~\ref{lem:finite_system}(a)) $\gamma(T)$ is self-adjoint and non-negative, Proposition~\ref{prop:local} with $\delta=\bar Y$ and the same $T=T(\bar Y)=1/(8C\bar Y)$ applied to datum $\gamma(T)$ yields a unique $\tilde\gamma\in C([T,2T];H^1\Sp^1)$ with $\tilde\gamma(T)=\gamma(T).$ The concatenation satisfies~\eqref{eq:Duhamel} on $[0,2T]$ with datum $\gamma_0,$ and self-adjointness and non-negativity are propagated across the join by Lemma~\ref{lem:finite_system}(a) applied to the concatenated solution.

Crucially, Lemma~\ref{lem:finite_system}(b) applied to the concatenated solution on $[0,2T]$ gives $\|\gamma(t)\|_{H^1\Sp^1}\le\bar Y$ for all $t\in[0,2T]$ with the \emph{same} $\bar Y,$ determined by $(\|\gamma_0\|_{H^1\Sp^1},p,q,|\T|)$ and not by $\gamma(T).$ Hence $\|\gamma(2T)\|_{H^1\Sp^1}\le\bar Y$ and the argument may be iterated. By induction, the solution exists on $[0,nT]$ for every $n\in\mathbb{N},$ i.e.\ on $[0,\infty).$ The argument on $(-\infty,0]$ is identical by time reversal.

\medskip

\noindent\textbf{Step 4: Uniqueness.} For any $\tilde\gamma\in C(\R;H^1\Sp^1)$ with $\tilde\gamma(0)=\gamma_0,$ continuity at $t=0$ places $\tilde\gamma$ in the uniqueness ball of Proposition~\ref{prop:local} on a sufficiently short initial interval; iterating in time (and on the negative axis) yields $\tilde\gamma=\gamma$ on $\R.$

\medskip

\noindent\textbf{Step 5: Higher Sobolev regularity (clause (b)).} Assume in addition $\gamma_0\in H^s\Sp^1$ for some $s\ge 2.$ Setting $s'=s-2\ge 0,$ Lemma~\ref{lem:finite_system}(c) applied with this $s'$ yields, on $[0,nT]$ for every $n\in\mathbb{N},$
\[
  \gamma\in C([0,nT];H^s\Sp^1)\cap C^1([0,nT];H^{s-2}\Sp^1),
\]
with the bound
\[
  \|\gamma(t)\|_{H^s\Sp^1}\;\le\;\|\gamma_0\|_{H^s\Sp^1}\,e^{C_{s'}|q|\,\bar Y\,|t|},
\]
and~\eqref{eq:Alber_op} holding in $H^{s-2}\Sp^1.$ Letting $n\to\infty$ and using the time-symmetric argument extends to $t\in\R,$ which proves clause (b).
\qed

\medskip

\section{On Penrose-stable backgrounds}\label{sec:stab}

We now consider the Alber equation around a nonzero homogeneous background $\Gamma$ on~$\T$. Theorem~\ref{thm:main} applies to $\Gamma$ and to perturbations $\Gamma+U_0$ in both the focusing and defocusing cases, so the datum-to-existence part of Theorem~\ref{thm:nlstab} is automatic. The Penrose-stability hypothesis only enters in the slow growth of $\|\gamma(t)-\Gamma\|_{H^1\Sp^1}.$

If $\Gamma$ is linearly stable in the sense of the Penrose condition (cf.\ Definition~\ref{def:penrose}), one can combine the a priori control of Theorem~\ref{thm:main} with a quantitative bound on the linearized propagator $e^{tL_\Gamma}$ (Proposition~\ref{prop:propagator} and Corollary~\ref{cor:poly_propagator}) to establish polynomial-in-$t$ control of the perturbation on the algebraic time window $|t|\le c_\Gamma\,\varepsilon^{-1/5}.$ A bootstrap argument closes the proof.

\subsection{Setup and linear stability}\label{subsec:setup_linstab}

Let $\Gamma$ be a self-adjoint, non-negative, translation-invariant operator on $L^2(\T)$ with Fourier symbol $\hat\Gamma(n)\ge 0$, $n\in\Z;$
equivalently, its integral kernel is
\begin{equation}
\label{eq:def_GammaFC}
	K_\Gamma(x,y) \;=\; \frac{1}{2\pi}\sum_{n\in\Z}\hat\Gamma(n)\,e^{in(x-y)},
\end{equation}
 The hypothesis $\Gamma\in H^1\Sp^1$ of Theorem~\ref{thm:nlstab} already ensures $\hat\Gamma\in\ell^1(\Z)$.
As mentioned in the Introduction, and can be checked directly, $\Gamma$ is a homogeneous solution of \eqref{eq:Alber_op}.

Now consider the perturbation $\gamma=\Gamma+U.$ The operator valued version of \eqref{eq:Alber1} is
\begin{equation}\label{eq:stab}
  i\partial_t U
  = \underbrace{p[\Delta,U]
  + q[V_{\rho_U},\Gamma]
  \vphantom{\big|}}_{L_\Gamma(U)}
  \;+\;
  \underbrace{q[V_{\rho_U},U]
  \vphantom{\big|}}_{Q(U)}.
\end{equation}
The quadratic part $Q(U)$ is the  nonlinear self-interaction, which can be controlled in $L^\infty_tH^1\Sp^1$ by Proposition~\ref{prop:bilinear}. Nonlinear stability of $\Gamma$ therefore reduces to bounding the propagator of the linear part $L_\Gamma$ in $H^1\Sp^1$. Dropping the quadratic part, we obtain the operator-valued linearized equation
\begin{equation}\label{eq:stab_lin}
  i\partial_t U
  = p[\Delta,U]
  + q[V_{\rho_U},\Gamma]
  \;=\; L_\Gamma(U),
\end{equation}
the operator-valued counterpart of~\eqref{eq:Alber1lin}.

On the Fourier side,~\eqref{eq:stab_lin} becomes
\begin{equation}\label{eq:stab_fourier}
  i\partial_t U_{mn}
  = -p(m^2-n^2)\,U_{mn}
  - \frac{q}{2\pi}\bigl(\hat\Gamma(m)
  -\hat\Gamma(n)\bigr)\,
  \hat\rho_U(m{-}n),
\end{equation}
where $\hat\rho_U(k) :=\sum_j U_{j+k,\,j}$ encodes the $k$-th diagonal of $U$ in the Fourier basis $e_n=(2\pi)^{-1/2}e^{inx},$ and equals $(2\pi)^{1/2}$ times the unitary Fourier coefficient of the density $\rho_U.$ The factor $(2\pi)^{-1}$ in the nonlinear term originates from $(V_f)_{mn}=(2\pi)^{-1}\hat\rho_U(m-n)$ in this basis. Solving~\eqref{eq:stab_fourier} by Duhamel and summing over the $k$-th diagonal ($m=j+k$, $n=j$) yields, for each $k\in\Z\setminus\{0\}$, a scalar Volterra equation:
\begin{equation}\label{eq:volterra}
  \hat\rho_U(k,t)
  = \hat\rho_U^{\mathrm{free}}(k,t)
  + \frac{iq}{2\pi}\int_0^t\Phi_k(t{-}s)\,
  \hat\rho_U(k,s)\,ds,
\end{equation}
with free density
\begin{equation}\label{eq:free_density}
  \hat\rho_U^{\mathrm{free}}(k,t)
  = \sum_{j\in\Z}
  e^{ipk(2j+k)t}\,U_{j+k,\,j}(0)
\end{equation}
and Volterra kernel
\begin{equation}\label{eq:volterra_kernel}
  \Phi_k(\tau)
  = \sum_{j\in\Z}
  \bigl(\hat\Gamma(j{+}k)
  -\hat\Gamma(j)\bigr)\,
  e^{ipk(2j+k)\tau}.
\end{equation}
For $k=0$, $\hat\rho_U(0)=\tr(U)$ is conserved. The crucial point is that equation~\eqref{eq:volterra} involves only the density $\hat\rho_U$, not the full operator $U$: the infinite-dimensional operator problem reduces to a countable family of scalar Volterra equations.

Now we have enough to state the following
\begin{definition}[Penrose stability
on $\T$]\label{def:penrose}
Let $\Gamma$ be a self-adjoint, non-negative, translation invariant operator on $L^2(\mathbb{T}),$ with $\Gamma\in H^1\Sp^1.$ Note that for such a $\Gamma,$ the singular values are exactly $\{|\hat\Gamma(n)|\}_{n\in\Z},$ so $\|\hat\Gamma\|_{\ell^1}=\|\Gamma\|_{\Sp^1}\le\|\Gamma\|_{H^1\Sp^1}$; in particular $\hat\Gamma\in\ell^1(\Z)$ automatically. The homogeneous background $\Gamma$ will be called \emph{Penrose-stable} with constant $\kappa>0$ if, for every $k\in\Z\setminus\{0\}$,
\begin{equation}\label{eq:penrose}
  \inf_{\operatorname{Re}\lambda> 0}
  \Bigl|1-\frac{iq}{2\pi}\,\tilde\Phi_k(\lambda)\Bigr|
  \ge\kappa,
\end{equation}
where $\tilde\Phi_k(\lambda) =\int_0^\infty e^{-\lambda\tau} \Phi_k(\tau)\,d\tau$ is the Laplace transform of the Volterra kernel~\eqref{eq:volterra_kernel}, defined for $\operatorname{Re}\lambda>0.$
\end{definition}

This is equivalent to the stability criterion of~\cite{AthanassoulisKyza2024}; the idea is that condition~\eqref{eq:penrose} ensures that no Fourier mode of the density can grow exponentially. On~$\mathbb{R},$ the Volterra kernel $\Phi_k$ is constructed from a compactly supported power spectrum, and its Laplace transform $\tilde\Phi_k$ can be expressed as a (rescaled) Hilbert transform of a compactly supported function with integral zero. This zero-integral property, together with compact support, implies that the boundary trace $\lim_{\eta\to 0^+}\tilde\Phi_k(\eta+is)$ is in $L^1_s(\mathbb{R})$~\cite{AAPS2020}; consequently, the Bromwich contour can be taken at $\operatorname{Re}\lambda=\eta\to 0^+$ and one obtains non-growth of the linearized solution.

On the torus, the Volterra kernel~\eqref{eq:volterra_kernel} is an almost-periodic function,  and the Laplace transform $\tilde\Phi_k(\lambda)$ has poles along the entire imaginary axis;  the boundary trace as $\operatorname{Re}\lambda\to 0^+$ genuinely is not an $L^1$  function. Thus, we cannot take the limit $\operatorname{Re}\lambda=\eta\to 0^+$ of the Bromwich contour.
Instead, we use a regularized Bromwich contour $\{\operatorname{Re}\lambda=\eta\}$ with $\eta>0,$ at the cost of a prefactor $C_\Gamma(\eta)e^{\eta t}$ in the propagator estimate (cf.\ Proposition~\ref{prop:propagator}). Optimizing this prefactor over $\eta>0$ \emph{pointwise in $t$} — with $\eta=\eta(t)\sim 1/\langle t\rangle$ — converts the $e^{\eta t}$ growth into the polynomial-in-$t$ propagator bound of Corollary~\ref{cor:poly_propagator}, which drives the algebraic time window in Theorem~\ref{thm:nlstab}.

\begin{remark}[An explicit Penrose-unstable background]
\label{rem:explicit_unstable}
We exhibit a $\Gamma\in H^1\Sp^1(\T)$ that fails the Penrose condition~\eqref{eq:penrose}. For simplicity take $p=q=1,$ and define $\Gamma$ by its Fourier coefficients
\begin{equation}\label{eq:cntrxpl}
  \hat\Gamma(0)=2\pi,\qquad\hat\Gamma(n)=0\text{ for }n\ne 0,
\end{equation}
i.e.\ $\Gamma$ is $2\pi$ times the rank-one projection onto the constant mode. Then $\|\Gamma\|_{H^1\Sp^1}=\sum_{n\in\Z}\langle n\rangle^2\hat\Gamma(n)=2\pi.$ For $k=1,$ the Volterra kernel~\eqref{eq:volterra_kernel} reduces to
\[
  \Phi_1(\tau)
  =\sum_{j\in\Z}\bigl(\hat\Gamma(j{+}1)-\hat\Gamma(j)\bigr)e^{ip(2j+1)\tau}
  = 2\pi(e^{-i\tau}-e^{i\tau})
  =-4\pi i\sin\tau,
\]
so, for $\operatorname{Re}\lambda>0,$
\[
  \tilde\Phi_1(\lambda)=\int_0^\infty e^{-\lambda\tau}\Phi_1(\tau)\,d\tau=-\frac{4\pi i}{\lambda^2+1},
\]
and hence
\[
  1-\frac{iq}{2\pi}\,\tilde\Phi_1(\lambda)
  =1-\frac{2}{\lambda^2+1}
  =\frac{\lambda^2-1}{\lambda^2+1}.
\]
This vanishes at $\lambda=\pm 1.$ The Volterra resolvent $\bigl(1-(iq/(2\pi))\tilde\Phi_1\bigr)^{-1}$ has a genuine pole at $\lambda=1,$ which produces a linearly unstable mode growing like $e^t.$ 
\end{remark}

\begin{proposition}[Propagator estimate]
\label{prop:propagator}
Let $\Gamma$ be Penrose-stable with $\Gamma\in H^1\Sp^1(\T).$ Then for every $\eta>0$ the linearized propagator satisfies
\begin{equation}\label{eq:propagator_bound}
  \|e^{tL_\Gamma}U_0\|_{H^1\Sp^1}
  \le C_\Gamma(\eta)\,e^{\eta|t|}\,
  \|U_0\|_{H^1\Sp^1},
  \qquad t\in\R.
\end{equation}

Moreover, the non-constant Fourier modes of the density of the linearized solution satisfy the weighted $L^2_t$ bound
\begin{equation}\label{eq:density_propagator}
  \Bigl(\sum_{k\ne 0}\langle k\rangle^2
  \int_0^\infty e^{-2\eta t}
  \bigl|\hat\rho_{e^{tL_\Gamma}U_0}(k,t)\bigr|^2\,dt
  \Bigr)^{1/2}
  \le \frac{C}{\kappa\sqrt{\eta}}\,
  \|U_0\|_{H^1\Sp^1}.
\end{equation}
\end{proposition}

\begin{remark}
The $k=0$ Fourier mode is conserved, $\hat\rho_{e^{tL_\Gamma}U_0}(0,t)=\tr U_0$, and plays no role in the nonlinear stability argument since $[V_c,\Gamma]=0$ for any constant~$c$.

The constant $C_\Gamma(\eta)$ diverges as $\eta\to 0^+,$ $\lim\limits_{\eta\to0^+} C_\Gamma(\eta) = +\infty,$ thus $\eta$ cannot be taken to tend to zero.
\end{remark}

\begin{proof}
We argue for $t\ge 0$; the case $t<0$ is identical by time reversal. The central object is the Fourier coefficient of the density, $\hat\rho_U(k,t)$, which for each $k\in\Z\setminus\{0\}$ satisfies the scalar Volterra equation
\begin{equation}\label{eq:volterra_proof}
  \hat\rho_U(k,t)
  = \hat\rho_U^{\mathrm{free}}(k,t)
  + \frac{iq}{2\pi}\int_0^t\Phi_k(t{-}s)\,
  \hat\rho_U(k,s)\,ds,
\end{equation}
with forcing bounded by
\begin{equation}\label{eq:Mk_def}
  |\hat\rho_U^{\mathrm{free}}(k,t)|
  \le M_k
  :=\sum_{j\in\Z}|U_{j+k,j}(0)|.
\end{equation}
The proof has four steps. Step~1 establishes a resolvent identity for the Laplace transform $\tilde{\hat\rho}_U(k,\lambda)$ on the right half-plane. Step~2 converts this into the weighted $L^2_t$ density bound~\eqref{eq:density_propagator} via Plancherel on the shifted contour $\operatorname{Re}\lambda=\eta$. Step~3 upgrades the $L^2_t$ bound to a pointwise-in-$t$ bound on $\|\rho_U(t)\|_{H^1}$. Step~4 closes the loop through Duhamel to obtain the operator bound~\eqref{eq:propagator_bound}. The technical inputs are the two properties of $\Phi_k$ (cf.\ Lemma~\ref{lem:kernel_L1}), the $L^2$ Paley--Wiener theorem on a shifted half-plane (cf.\ Lemma~\ref{lem:PW_shifted}), and the Fourier summation estimate (cf.\ Lemma~\ref{lem:fourier_sum}), all collected in Appendix~\ref{app:volterra}.

\medskip

\noindent\textbf{Step 1: Resolvent identity on $\{\operatorname{Re}\lambda>0\}$.} The pointwise bound $|\Phi_k(\tau)|\le 2\|\hat\Gamma\|_{\ell^1}$ from Lemma~\ref{lem:kernel_L1}(i), combined with~\eqref{eq:Mk_def}, gives via a standard Gr\"onwall argument on~\eqref{eq:volterra_proof} the a priori bound
\begin{equation}\label{eq:apriori_gronwall}
  |\hat\rho_U(k,t)|\le M_k\,e^{C_0 t},
  \qquad
  C_0:=\frac{|q|\,\|\hat\Gamma\|_{\ell^1}}{\pi}.
\end{equation}
In particular
\[
  \tilde{\hat\rho}_U(k,\lambda)
  :=\int_0^\infty e^{-\lambda t}
  \hat\rho_U(k,t)\,dt
\]
converges absolutely for $\operatorname{Re}\lambda>C_0$, and taking Laplace transforms in~\eqref{eq:volterra_proof} produces
\begin{equation}\label{eq:volterra_solved}
  \tilde{\hat\rho}_U(k,\lambda)
  = \frac{\tilde{\hat\rho}_U^{\mathrm{free}}
  (k,\lambda)}
  {1-(iq/(2\pi))\,\tilde\Phi_k(\lambda)}.
\end{equation}
Both sides of~\eqref{eq:volterra_solved} are analytic on the larger set $\{\operatorname{Re}\lambda>0\}$: the numerator by boundedness of $\hat\rho_U^{\mathrm{free}}(k,\cdot)$, the denominator by Lemma~\ref{lem:kernel_L1}(ii) and the Penrose lower bound~\eqref{eq:penrose} $|1-(iq/(2\pi))\tilde\Phi_k(\lambda)|\ge\kappa$. By analytic continuation, \eqref{eq:volterra_solved} holds on the whole right half-plane, and the left-hand side extends analytically from $\{\operatorname{Re}\lambda>C_0\}$ to $\{\operatorname{Re}\lambda>0\}$.

\medskip

\noindent\textbf{Step 2: Weighted $L^2_t$ bound on the position density.} Fix $\eta>0$. From the Penrose condition and~\eqref{eq:volterra_solved},
\begin{equation}\label{eq:F_bound_vertical}
  |\tilde{\hat\rho}_U(k,\sigma+is)|
  \le \kappa^{-1}\,
  |\tilde{\hat\rho}_U^{\mathrm{free}}
  (k,\sigma+is)|,
  \qquad
  \sigma>0,\ s\in\R.
\end{equation}
For the free density, fix $\sigma\ge\eta$ and set
\[
  g_\sigma(t)
  := e^{-\sigma t}\,
  \hat\rho_U^{\mathrm{free}}(k,t)\,
  \mathbf 1_{t>0}.
\]
By~\eqref{eq:Mk_def}, $|g_\sigma(t)|\le M_k\,e^{-\sigma t} \mathbf 1_{t>0}$, so $g_\sigma\in L^2(\R)$ with
\begin{equation}\label{eq:gsigma_L2}
  \|g_\sigma\|_{L^2(\R)}^2
  = \int_0^\infty e^{-2\sigma t}\,
  |\hat\rho_U^{\mathrm{free}}(k,t)|^2\,dt
  \le M_k^2\int_0^\infty e^{-2\sigma t}\,dt
  = \frac{M_k^2}{2\sigma}
  \le \frac{M_k^2}{2\eta}.
\end{equation}
The Fourier transform of $g_\sigma$ at frequency $s\in\R$ is
\[
  \widehat{g_\sigma}(s)
  = \int_\R e^{-ist}g_\sigma(t)\,dt
  = \int_0^\infty e^{-(\sigma+is)t}\,
  \hat\rho_U^{\mathrm{free}}(k,t)\,dt
  = \tilde{\hat\rho}_U^{\mathrm{free}}
  (k,\sigma+is),
\]
i.e.\ the vertical slice of the Laplace transform at $\operatorname{Re}\lambda=\sigma$. Plancherel $\|\widehat{g_\sigma}\|_{L^2(\R)}^2 =2\pi\|g_\sigma\|_{L^2(\R)}^2$, combined with~\eqref{eq:gsigma_L2}, gives
\begin{equation}\label{eq:free_L2_vertical}
  \frac{1}{2\pi}\int_\R
  |\tilde{\hat\rho}_U^{\mathrm{free}}
  (k,\sigma+is)|^2\,ds
  = \int_0^\infty e^{-2\sigma t}
  |\hat\rho_U^{\mathrm{free}}(k,t)|^2\,dt
  \le \frac{M_k^2}{2\eta}
\end{equation}
for every $\sigma\ge\eta$. Combining \eqref{eq:F_bound_vertical} and~\eqref{eq:free_L2_vertical},
\begin{equation}\label{eq:F_L2_vertical}
  \sup_{\sigma>\eta}
  \int_\R
  |\tilde{\hat\rho}_U(k,\sigma+is)|^2\,ds
  \le \frac{\pi M_k^2}{\kappa^2\eta}.
\end{equation}
Observe that the assumptions of Lemma~\ref{lem:PW_shifted} hold: the exponential-type bound~\eqref{eq:apriori_gronwall} ensures that the Laplace transform $\tilde{\hat\rho}_U(k,\cdot)$ of $\hat\rho_U(k,\cdot)$ converges absolutely and is analytic on $\{\operatorname{Re}\lambda>C_0\}$; it extends analytically to the larger half-plane $\{\operatorname{Re}\lambda>\eta\}$ by Step~1; and it satisfies the bound~\eqref{eq:F_L2_vertical}, which is uniform in $\sigma=\operatorname{Re}\lambda$ over $\sigma>\eta$. Lemma~\ref{lem:PW_shifted} now yields $e^{-\eta t}\hat\rho_U(k,t)\in L^2(\R_+)$ and, together with~\eqref{eq:F_L2_vertical},
\begin{equation}\label{eq:density_weighted_k}
  \int_0^\infty e^{-2\eta t}
  |\hat\rho_U(k,t)|^2\,dt
  = \frac{1}{2\pi}
  \int_\R|\tilde{\hat\rho}_U(k,\eta+is)|^2\,ds
  \le \frac{M_k^2}{2\kappa^2\eta}.
\end{equation}
Multiplying by $\langle k\rangle^2$, summing over $k\ne 0$, and invoking the Fourier summation estimate $\sum_k\langle k\rangle^2 M_k^2 \le C\|U_0\|_{H^1\Sp^1}^2$ from Lemma~\ref{lem:fourier_sum}, Fubini gives precisely
\begin{equation}\label{eq:density_weighted_final}
  \sum_{k\ne 0}\langle k\rangle^2
  \int_0^\infty e^{-2\eta t}
  |\hat\rho_U(k,t)|^2\,dt
  \le \frac{C}{\kappa^2\eta}\,
  \|U_0\|_{H^1\Sp^1}^2,
\end{equation}
which is~\eqref{eq:density_propagator}.

\medskip

\noindent\textbf{Step 3: Pointwise bound on $\|\rho_U(t)\|_{H^1}$.} We will use the $L^2_t$ bound of the position density that we showed in Step~2 to show a pointwise (i.e., $L^\infty_t$) bound for  $\|\rho_U(t)\|_{H^1};$ we will use the Volterra equation~\eqref{eq:volterra_proof} to do that.

For fixed $t\ge 0$, factor $e^{\eta t}$ out of the convolution:
\[
  \int_0^t\Phi_k(t{-}s)\,
  \hat\rho_U(k,s)\,ds
  = e^{\eta t}\!
  \int_0^t
  \bigl[e^{-\eta(t-s)}\Phi_k(t{-}s)\bigr]
  \bigl[e^{-\eta s}\hat\rho_U(k,s)\bigr]ds.
\]
We apply Cauchy--Schwarz in $s$ to the bracketed integral:
\begin{equation}\label{eq:CS_step3}
  \Bigl|\!\int_0^t
  e^{-\eta(t-s)}\Phi_k(t{-}s)\cdot
  e^{-\eta s}\hat\rho_U(k,s)\,ds\Bigr|
  \le A_k(t)\cdot B_k(t),
\end{equation}
where
\[
  A_k(t)
  := \Bigl(\int_0^t
  e^{-2\eta(t-s)}|\Phi_k(t-s)|^2\,ds\Bigr)^{\!1/2},
\]
\[
  B_k(t)
  := \Bigl(\int_0^t
  e^{-2\eta s}|\hat\rho_U(k,s)|^2\,ds\Bigr)^{\!1/2}.
\]
The first factor $A_k(t)$ is bounded by a pure kernel integral: changing variables $\tau=t-s$ and extending the domain to $[0,\infty)$,
\[
  A_k^2(t)
  = \int_0^t e^{-2\eta\tau}
  |\Phi_k(\tau)|^2\,d\tau
  \le \int_0^\infty e^{-2\eta\tau}
  |\Phi_k(\tau)|^2\,d\tau.
\]
Using the sup bound $|\Phi_k(\tau)|\le 2\|\hat\Gamma\|_{\ell^1}$ from Lemma~\ref{lem:kernel_L1}(i),
\[
  \int_0^\infty e^{-2\eta\tau}
  |\Phi_k(\tau)|^2\,d\tau
  \le (2\|\hat\Gamma\|_{\ell^1})^2
  \int_0^\infty e^{-2\eta\tau}d\tau
  = \frac{2\,\|\hat\Gamma\|_{\ell^1}^2}{\eta}.
\]
Hence
\begin{equation}\label{eq:Ak_bound}
  A_k(t)
  \le \frac{\sqrt{2}\,\|\hat\Gamma\|_{\ell^1}}
  {\sqrt{\eta}}.
\end{equation}
The second factor $B_k(t)$ is controlled by the weighted $L^2_t$ bound~\eqref{eq:density_weighted_k} established in Step~2: extending the domain from $[0,t]$ to $[0,\infty)$,
\begin{equation}\label{eq:Bk_bound}
  B_k(t)^2
  \le \int_0^\infty e^{-2\eta s}
  |\hat\rho_U(k,s)|^2\,ds
  \le \frac{M_k^2}{2\kappa^2\eta},
  \qquad\text{i.e.}\qquad
  B_k(t)\le \frac{M_k}{\kappa\sqrt{2\eta}}.
\end{equation}
Thus, multiplying~\eqref{eq:Ak_bound} and~\eqref{eq:Bk_bound} together, \eqref{eq:CS_step3} becomes
\[
  \Bigl|\!\int_0^t
  e^{-\eta(t-s)}\Phi_k(t{-}s)\cdot
  e^{-\eta s}\hat\rho_U(k,s)\,ds\Bigr|
  \le \frac{\sqrt{2}\,\|\hat\Gamma\|_{\ell^1}}
  {\sqrt{\eta}}\cdot
  \frac{M_k}{\kappa\sqrt{2\eta}}
  = \frac{\|\hat\Gamma\|_{\ell^1}\,M_k}
  {\kappa\,\eta}.
\]
Restoring the factor $e^{\eta t}$ from the convolution identity above and multiplying by $|q|/(2\pi)$ we essentially reconstitute \eqref{eq:volterra}, leading to
\begin{equation}\label{eq:density_pointwise_k}
  |\hat\rho_U(k,t)-\hat\rho_U^{\mathrm{free}}(k,t)|
  = \Bigl|\frac{|q|}{2\pi}
  \int_0^t\Phi_k(t{-}s)\,
  \hat\rho_U(k,s)\,ds\Bigr|
  \le \frac{|q|\,\|\hat\Gamma\|_{\ell^1}}
  {2\pi\,\kappa\,\eta}\,e^{\eta t}\,M_k.
\end{equation}
Since $|\hat\rho_U^{\mathrm{free}}(k,t)| \le M_k\le M_k\,e^{\eta t}$,
\[
  |\hat\rho_U(k,t)|
  \le \Bigl(1+\frac{|q|\,\|\hat\Gamma\|_{\ell^1}}
  {2\pi\,\kappa\eta}\Bigr)e^{\eta t}M_k.
\]
Multiplying by $\langle k\rangle$, summing over $k\ne 0$ via Cauchy--Schwarz and Lemma~\ref{lem:fourier_sum},
\begin{equation}\label{eq:rho_pointwise_final}
  \Bigl(\sum_{k\ne 0}\langle k\rangle^2
  |\hat\rho_U(k,t)|^2\Bigr)^{1/2}
  \le C\Bigl(1+\frac{|q|\,\|\hat\Gamma\|_{\ell^1}}
  {2\pi\,\kappa\eta}\Bigr)e^{\eta t}\,
  \|U_0\|_{H^1\Sp^1}.
\end{equation}

\medskip

\noindent\textbf{Step 4: Operator bound via Duhamel.} Duhamel applied to the linearized equation~\eqref{eq:stab_lin} gives
\[
  U(t) = S(t)U_0
  - iq\int_0^t S(t{-}s)\,
  [V_{\rho_U(s)},\Gamma]\,ds,
\]
where $S(t)$ of course is isometric on $H^1\Sp^1$. For any $f\in H^1(\T)$, writing $\langle D\rangle V_f\Gamma\langle D\rangle =(\langle D\rangle V_f\langle D\rangle^{-1}) \,(\langle D\rangle\Gamma\langle D\rangle)$ and applying the Schatten--H\"older inequality $\Sp^\infty\cdot\Sp^1\hookrightarrow\Sp^1$ together with Lemma~\ref{lem:conj} gives $\|V_f\Gamma\|_{H^1\Sp^1} \le C\|f\|_{H^1}\|\Gamma\|_{H^1\Sp^1}$. The reverse ordering $\|\Gamma V_f\|_{H^1\Sp^1}$ is identical by taking adjoints ($V_f^*=V_{\bar f}$, same $H^1$-norm). Combining via the triangle inequality,
\begin{equation}\label{eq:commbd_step4}
  \|[V_f,\Gamma]\|_{H^1\Sp^1}
  \le C\|f\|_{H^1}\|\Gamma\|_{H^1\Sp^1},
  \qquad f\in H^1(\T).
\end{equation}
Since the part of $f$ that corresponds to the zero Fourier mode, $V_{\hat f(0)}=\hat f(0)\cdot I$ commutes with $\Gamma$, it does not contribute to the commutator,
\begin{equation}\label{eq:comm_meanzero}
  [V_f,\Gamma] \;=\; [V_{f-\hat f(0)},\Gamma],
\end{equation}
so that applying~\eqref{eq:commbd_step4} to $f=\rho_{U(s)}-\hat\rho_{U(s)}(0)$  yields
\[
  \|U(t)\|_{H^1\Sp^1}
  \le \|U_0\|_{H^1\Sp^1}
  + C|q|\,\|\Gamma\|_{H^1\Sp^1}
  \int_0^t\Bigl(\sum_{k\ne 0}\langle k\rangle^2|\hat\rho_U(k,s)|^2\Bigr)^{1/2}\,ds.
\]
Inserting~\eqref{eq:rho_pointwise_final} and using $\int_0^t e^{\eta s}ds \le \eta^{-1}e^{\eta t}$ yields
\[
  \|U(t)\|_{H^1\Sp^1}
  \le \Bigl[1
  + \frac{C\,|q|\,
  \|\Gamma\|_{H^1\Sp^1}\,
  \bigl(1+|q|\|\hat\Gamma\|_{\ell^1}
  /(2\pi\kappa\eta)\bigr)}
  {\eta}\Bigr]\,e^{\eta t}\,
  \|U_0\|_{H^1\Sp^1}.
\]
Setting
\begin{equation}\label{eq:Cgamma_def}
  C_\Gamma(\eta)
  := 1
  + \frac{C\,|q|\,
  \|\Gamma\|_{H^1\Sp^1}\,
  \bigl(1+|q|\|\hat\Gamma\|_{\ell^1}
  /(2\pi\kappa\eta)\bigr)}
  {\eta}
\end{equation}
gives~\eqref{eq:propagator_bound}.
\end{proof}

\begin{corollary}[Polynomial-in-$t$ linear propagator bound]
\label{cor:poly_propagator}
Under the hypotheses of Proposition~\ref{prop:propagator}, the linearized propagator satisfies
\begin{equation}\label{eq:poly_propagator_bound}
  \|e^{tL_\Gamma}U_0\|_{H^1\Sp^1}
  \;\le\; C_\star(\Gamma)\,
  \langle t\rangle^2\,
  \|U_0\|_{H^1\Sp^1},
  \qquad t\in\R,
\end{equation}
with
\begin{equation}\label{eq:Cstar_def}
  C_\star(\Gamma)\;=\;3\,\bigl(1+A+AB\bigr),
  \qquad
  A=C\,|q|\,\|\Gamma\|_{H^1\Sp^1},
  \qquad
  B=\frac{|q|\,\|\hat\Gamma\|_{\ell^1}}{2\pi\,\kappa},
\end{equation}
where $A,B,C_\star$ depend only on $\Gamma,p,q,\kappa$ and $C$ is the absolute constant of~\eqref{eq:Cgamma_def}.
\end{corollary}

\begin{proof}
Fix $t\in\R$ and apply Proposition~\ref{prop:propagator} with the choice $\eta=1/\langle t\rangle>0$. By~\eqref{eq:Cgamma_def},
\[
  C_\Gamma(1/\langle t\rangle)
  \;=\; 1+A\langle t\rangle+AB\langle t\rangle^2
  \;\le\; \bigl(1+A+AB\bigr)\,\langle t\rangle^2,
\]
using $\langle t\rangle\ge 1$. Since $e^{\eta|t|}=e^{|t|/\langle t\rangle}\le e<3,$~\eqref{eq:propagator_bound} yields
\[
  \|e^{tL_\Gamma}U_0\|_{H^1\Sp^1}
  \le 3\bigl(1+A+AB\bigr)\,
  \langle t\rangle^2\,
  \|U_0\|_{H^1\Sp^1}.
\]
\end{proof}


\subsection{Proof of Theorem~\ref{thm:nlstab}}\label{subsec:proof_nlstab}

\noindent\textbf{Step 1: Global existence.} By hypothesis $\gamma_0=\Gamma+U_0\ge 0$ is self-adjoint with $\|\gamma_0\|_{H^1\Sp^1}\le\|\Gamma\|_{H^1\Sp^1}+\varepsilon.$ Theorem~\ref{thm:main} provides a unique global solution $\gamma\in C(\R;\,H^1\Sp^1),$ and the perturbation $U(t):=\gamma(t)-\Gamma\in C(\R;\,H^1\Sp^1)$ satisfies the perturbation equation in mild form
\begin{equation}\label{eq:nlstab_duhamel}
  U(t)\;=\; e^{tL_\Gamma}U_0
  \;+\;\int_0^t e^{(t-s)L_\Gamma}\,Q(U(s))\,ds,
  \qquad Q(U)\;=\; -iq\,[V_{\rho_U},U],
\end{equation}
where $L_\Gamma$ is the linear part of~\eqref{eq:stab} around $\Gamma$ and $Q$ is the quadratic remainder.

\medskip

\noindent\textbf{Step 2: Bootstrap ansatz.} By Corollary~\ref{cor:poly_propagator}, for every $V\in H^1\Sp^1$ and every $\tau\in\R,$
\begin{equation}\label{eq:poly_prop_inline}
  \|e^{\tau L_\Gamma}V\|_{H^1\Sp^1}
  \;\le\; C_\star\,\langle\tau\rangle^2\,
  \|V\|_{H^1\Sp^1},
\end{equation}
with $C_\star=C_\star(\Gamma)$ as in~\eqref{eq:Cstar_def}. By Proposition~\ref{prop:bilinear},
\begin{equation}\label{eq:Q_quad}
  \|Q(U)\|_{H^1\Sp^1}
  \;\le\; C_{\rm Q}\,\|U\|_{H^1\Sp^1}^2,
  \qquad
  C_{\rm Q}=C\,|q|,
\end{equation}
with the same absolute constant $C$ as in Proposition~\ref{prop:bilinear}. Consider the hypothesis
\begin{equation}\label{eq:bootstrap_ansatz}
  \|U(s)\|_{H^1\Sp^1}\;\le\; 2\,C_\star\,\langle s\rangle^2\,\varepsilon
  \qquad\text{for all }s\in[0,T],
\end{equation}
and let $T^\sharp\in(0,+\infty]$ be the supremum of $T\ge 0$ for which~\eqref{eq:bootstrap_ansatz} holds. Continuity of $U$ and $\|U(0)\|_{H^1\Sp^1}\le\varepsilon<2C_\star\varepsilon$ give $T^\sharp>0.$

\medskip

\noindent\textbf{Step 3: Closing the bootstrap.} Insert~\eqref{eq:poly_prop_inline} and~\eqref{eq:Q_quad} into~\eqref{eq:nlstab_duhamel}: for $t\in[0,T^\sharp),$
\begin{align}
  \|U(t)\|_{H^1\Sp^1}
  &\;\le\; C_\star\,\langle t\rangle^2\,\varepsilon
  + C_\star\,C_{\rm Q}
  \int_0^t\langle t-s\rangle^2\,
  \|U(s)\|_{H^1\Sp^1}^2\,ds \notag\\
  &\;\le\; C_\star\,\langle t\rangle^2\,\varepsilon
  + 4\,C_\star^{3}\,C_{\rm Q}\,\varepsilon^2
  \int_0^t\langle t-s\rangle^2\,\langle s\rangle^4\,ds.
  \label{eq:U_bootstrap_ineq}
\end{align}
For $t\ge 0$ the bracket integral is bounded by
\begin{equation}\label{eq:bracket_integral}
  \int_0^t\langle t-s\rangle^2\,\langle s\rangle^4\,ds
  \;\le\; \langle t\rangle^2\int_0^t\langle s\rangle^4\,ds
  \;\le\; \langle t\rangle^2\cdot\langle t\rangle^4\cdot t
  \;\le\; \langle t\rangle^7.
\end{equation}
Hence~\eqref{eq:U_bootstrap_ineq} gives
\[
  \|U(t)\|_{H^1\Sp^1}
  \;\le\; C_\star\,\langle t\rangle^2\,\varepsilon
  + 4\,C_\star^{3}\,C_{\rm Q}\,\varepsilon^2\,\langle t\rangle^7.
\]
The bootstrap~\eqref{eq:bootstrap_ansatz} is strictly improved to $\le(3/2)C_\star\,\langle t\rangle^2\,\varepsilon$ on the set where
\begin{equation}\label{eq:closure_condition}
  4\,C_\star^{3}\,C_{\rm Q}\,\varepsilon^2\,\langle t\rangle^7
  \;\le\; \tfrac12\,C_\star\,\langle t\rangle^2\,\varepsilon,
  \quad\text{i.e.}\quad
  \langle t\rangle^{5}
  \;\le\; \frac{1}{8\,C_\star^{2}\,C_{\rm Q}\,\varepsilon}.
\end{equation}
Setting
\begin{equation}\label{eq:cGamma_def}
  c_\Gamma
  \;:=\; \bigl(8\,C_\star^{2}\,C_{\rm Q}\bigr)^{-1/5},
\end{equation}
condition~\eqref{eq:closure_condition} is equivalent to $|t|\le c_\Gamma\,\varepsilon^{-1/5},$ which is $T_\star$ in~\eqref{eq:Tstar_def}. By continuity, $T^\sharp\ge T_\star,$ and~\eqref{eq:bootstrap_ansatz} holds throughout $[0,T_\star],$ which is~\eqref{eq:nlstab_bound}. The case $t\le 0$ is identical by time reversal. \qed


\appendix

\section{On the Volterra kernel and Penrose stability}
\label{app:volterra}

We collect the two properties of the Volterra kernel $\Phi_k$ defined in~\eqref{eq:volterra_kernel} that are used in the proof of Proposition~\ref{prop:propagator}: a pointwise bound in $t$, and an explicit series representation of its Laplace transform on $\{\operatorname{Re}\lambda>0\}$ (together with analyticity there).

On the torus, $\Phi_k$ is almost-periodic and generically \emph{not} in $L^1(\R_+)$, so the classical $L^1$ Volterra resolvent calculus 
does not apply directly; this is precisely why the proof of Proposition~\ref{prop:propagator} proceeds via Bromwich inversion on a shifted contour rather than through an $L^1$ resolvent kernel.

\begin{lemma}\label{lem:kernel_L1}
Let $\Gamma\in H^1\Sp^1(\T)$ be translation-invariant, and let $\Phi_k$ be the Volterra kernel~\eqref{eq:volterra_kernel}. Then for each $k\ne 0$:
\begin{enumerate}
\item[(i)] $\Phi_k\in L^\infty(\R_+)$
with
\begin{equation}\label{eq:phi_k_sup}
  \sup_{\tau\ge 0}|\Phi_k(\tau)|
  \le 2\|\hat\Gamma\|_{\ell^1}.
\end{equation}
\item[(ii)] The Laplace transform
$\tilde\Phi_k(\lambda) =\int_0^\infty e^{-\lambda\tau} \Phi_k(\tau)\,d\tau$ is analytic on $\{\operatorname{Re}\lambda>0\}$ and admits the absolutely convergent representation
\begin{equation}\label{eq:phi_k_laplace}
  \tilde\Phi_k(\lambda)
  = \sum_{j\in\Z}
  \frac{\hat\Gamma(j{+}k)-\hat\Gamma(j)}
  {\lambda-ipk(2j+k)},
  \qquad
  \operatorname{Re}\lambda>0.
\end{equation}
\end{enumerate}
Here $\hat\Gamma\in\ell^1(\Z)$ thanks to $\Gamma\in H^1\Sp^1(\T)\subset\Sp^1(\T)$ and translation invariance, which give $\|\hat\Gamma\|_{\ell^1} =\|\Gamma\|_{\Sp^1}$.
\end{lemma}

\begin{proof}
From~\eqref{eq:volterra_kernel}, $|\Phi_k(\tau)| \le\sum_j|\hat\Gamma(j{+}k)-\hat\Gamma(j)| \le 2\|\hat\Gamma\|_{\ell^1}$, proving~(i).

For (ii), the integral $\int_0^\infty e^{-\lambda\tau} e^{ipk(2j+k)\tau}\,d\tau =1/(\lambda-ipk(2j+k))$ is valid for $\operatorname{Re}\lambda>0$, and
\[
  \sum_{j\in\Z}
  \frac{|\hat\Gamma(j{+}k)-\hat\Gamma(j)|}
  {|\lambda-ipk(2j+k)|}
  \le
  \frac{2\|\hat\Gamma\|_{\ell^1}}
  {\operatorname{Re}\lambda},
\]
so term-by-term integration is justified and the series converges absolutely, proving the representation~\eqref{eq:phi_k_laplace}. Each term is analytic on $\{\operatorname{Re}\lambda>0\}$ and the bound above is locally uniform on any set $\{\operatorname{Re}\lambda\ge\eta\}$ with $\eta>0$, so by Weierstrass the sum is analytic on $\{\operatorname{Re}\lambda>0\}$.
\end{proof}

\medskip

We next record the version of the $L^2$ Paley--Wiener theorem used in the proof of Proposition~\ref{prop:propagator}. It is the standard statement
 specialized to a shifted right half-plane and phrased in the form we need: 

\begin{lemma}[$L^2$ Paley--Wiener on
a shifted contour]
\label{lem:PW_shifted}
Let $0<\eta<C_0$ and let $f\colon\R_+\to\C$ be locally integrable with $e^{-C_0 t}f\in L^\infty(\R_+)$, so that its Laplace transform
\[
  \tilde f(\lambda)
  :=\int_0^\infty e^{-\lambda t}
  f(t)\,dt
\]
converges absolutely and is analytic on $\{\operatorname{Re}\lambda>C_0\}$. Suppose that $\tilde f$ extends analytically to $\{\operatorname{Re}\lambda>\eta\}$ with
\begin{equation}\label{eq:PW_hyp}
  M:=\sup_{\sigma>\eta}
  \int_\R|\tilde f(\sigma+is)|^2\,ds
  <\infty.
\end{equation}
Then $e^{-\eta t}f\in L^2(\R_+)$ with
\begin{equation}\label{eq:PW_conclusion}
  \int_0^\infty e^{-2\eta t}
  |f(t)|^2\,dt
  = \frac{1}{2\pi}
  \int_\R |\tilde f(\eta+is)|^2\,ds
  \le \frac{M}{2\pi}.
\end{equation}
\end{lemma}

\medskip

Finally, the following elementary Fourier summation lemma converts the $\Sp^1$-trace summability of the initial data into the weighted $\ell^2$ bound needed to sum the mode-by-mode density estimates into an $H^1$-bound.

\begin{lemma}[Fourier summation of
trace coefficients]
\label{lem:fourier_sum}
For any $U_0\in H^1\Sp^1(\T)$ with matrix entries $U_{m,j}(0) =\langle e_m,U_0 e_j\rangle$ in the Fourier basis, the following bound holds
\begin{equation}\label{eq:fourier_sum_bound}
  \sum_{k\ne 0}\langle k\rangle^2
  \Bigl(\sum_{j\in\Z}
  |U_{j+k,j}(0)|\Bigr)^{\!2}
  \le C\,\|U_0\|_{H^1\Sp^1}^2.
\end{equation}
\end{lemma}

\begin{proof}
Apply Cauchy--Schwarz in $j$ with weight $1/(\langle j\rangle\langle j+k\rangle)$:
\begin{equation}\label{eq:fc112}
  \Bigl(\sum_j |U_{j+k,j}(0)|\Bigr)^{\!2}
  \le \Bigl(\sum_j \langle j+k\rangle^2
  \langle j\rangle^2|U_{j+k,j}(0)|^2\Bigr)\,
  \Bigl(\sum_j \frac{1}{\langle j+k\rangle^2
  \langle j\rangle^2}\Bigr).
\end{equation}
Split at $|j|\lessgtr|k|/2.$ On $|j|\le|k|/2,$ $|j+k|\ge|k|/2$ gives $1/\langle j+k\rangle\le2/\langle k\rangle,$ so
\[
  \sum_{|j|\le|k|/2}\frac{1}{\langle j+k\rangle^2\,\langle j\rangle^2}
  \;\le\; \frac{4}{\langle k\rangle^2}\sum_{j\in\Z}\frac{1}{\langle j\rangle^2}
  \;=\; \frac{C}{\langle k\rangle^2}.
\]
Symmetrically, on $|j|>|k|/2,$ we have $1/\langle j\rangle\le2/\langle k\rangle,$ so
\[
  \sum_{|j|>|k|/2}\frac{1}{\langle j+k\rangle^2\,\langle j\rangle^2}
  \;\le\; \frac{4}{\langle k\rangle^2}\sum_{j\in\Z}\frac{1}{\langle j+k\rangle^2}
  \;=\;\frac{4}{\langle k\rangle^2}\sum_{m\in\Z}\frac{1}{\langle m\rangle^2}
  \;=\; \frac{C}{\langle k\rangle^2},
\]
Thus finally (changing the value of $C$)
\[
  \sum_{j\in\Z}\frac{1}{\langle j+k\rangle^2\,\langle j\rangle^2}
  \;\le\; \frac{C}{\langle k\rangle^2},
\]
so \eqref{eq:fc112} yields 
\[
\langle k\rangle^2\bigl(\sum_j|U_{j+k,j}(0)|\bigr)^2 \le C\sum_j\langle j+k\rangle^2\langle j\rangle^2|U_{j+k,j}(0)|^2.
\]
Summing over $k$ and relabeling $m=j+k$:
\begin{align*}
  \sum_{k\ne 0}
  \langle k\rangle^2
  \Bigl(\sum_j|U_{j+k,j}(0)|\Bigr)^{\!2}
  &\le C\sum_{m,j}\langle m\rangle^2
  \langle j\rangle^2|U_{m,j}(0)|^2\\
  &= C\,\|\langle D\rangle U_0\langle D\rangle\|_{\Sp^2}^2
  \le C\,\|U_0\|_{H^1\Sp^1}^2,
\end{align*}
using the Schatten embedding $\Sp^1\hookrightarrow\Sp^2$ in the last step.
\end{proof}


\begin{thebibliography}{99}

\bibitem{Alber1978}
I.~E.~Alber, \emph{The effects of randomness on the stability of two-dimensional surface wavetrains}, Proc.\ Roy.\ Soc.\ London Ser.~A \textbf{363} (1978), no.~1715, 525--546.

\bibitem{AndradeStiassnie2020}
D.~Andrade and M.~Stiassnie, \emph{New solutions of the CSY equation reveal increases in freak wave occurrence}, Wave Motion \textbf{97} (2020), 102581.

\bibitem{AAPS2020}
A.~Athanassoulis, G.~A.~Athanassoulis, M.~Ptashnyk, and T.~Sapsis, \emph{Strong solutions for the Alber equation and stability of unidirectional wave spectra}, Kinet.\ Relat.\ Models \textbf{13} (2020), no.~4, 703--737.

\bibitem{AthanassoulisGramstad2021}
A.~G.~Athanassoulis and O.~Gramstad, \emph{Modelling of ocean waves with the Alber equation: application to non-parametric spectra and generalisation to crossing seas}, Fluids \textbf{6} (2021), no.~8, 291.

\bibitem{AthanassoulisKarakatsaniKyza2025}
A.~Athanassoulis, F.~Karakatsani, and I.~Kyza, \emph{A structure-preserving, second-order-in-time scheme for the von Neumann equation with power nonlinearity}, arXiv:2506.06879, 2025.

\bibitem{AthanassoulisKyza2024}
A.~Athanassoulis and I.~Kyza, \emph{Modulation instability and convergence of the random-phase approximation for stochastic sea states}, Water Waves \textbf{6} (2024), no.~1, 145--167.


\bibitem{BoveDaPratoFano1976}
A.~Bove, G.~Da Prato, and G.~Fano, \emph{On the Hartree--Fock time-dependent problem}, Comm.\ Math.\ Phys.\ \textbf{49} (1976), 25--33.

\bibitem{Brezis2011}
H.~Brezis, \emph{Functional Analysis, Sobolev Spaces and Partial Differential Equations}, Universitext, Springer, New York, 2011.




\bibitem{Cazenave2003}
T.~Cazenave, \emph{Semilinear Schr\"odinger Equations}, Courant Lecture Notes in Mathematics, vol.~10, American Mathematical Society, Providence, RI, 2003.

\bibitem{ChenHongPavlovic2017}
T.~Chen, Y.~Hong, and N.~Pavlovi\'c, \emph{Global well-posedness of the NLS system for infinitely many fermions}, Arch.\ Ration.\ Mech.\ Anal.\ \textbf{224} (2017), no.~1, 91--123.

\bibitem{CrawfordSaffmanYuen1980}
D.~R.~Crawford, P.~G.~Saffman, and H.~C.~Yuen, \emph{Evolution of a random inhomogeneous field of nonlinear deep-water gravity waves}, Wave Motion \textbf{2} (1980), no.~1, 1--16.

\bibitem{Dysthe2008}
K.~B.~Dysthe, H.~E.~Krogstad, and P.~M\"uller, \emph{Oceanic rogue waves}, Annu.\ Rev.\ Fluid Mech.\ \textbf{40} (2008), 287--310.

\bibitem{MeiStiassnieYue2005b}
C.~C. Mei, M.~Stiassnie, and D.~K.-P. Yue, \emph{Theory and Applications of Ocean Surface Waves: Part 2: Nonlinear Aspects}, World Scientific, Singapore, 2005.



\bibitem{GrafakosOh2014}
L.~Grafakos and S.~Oh, \emph{The Kato--Ponce inequality}, Comm.\ Partial Differential Equations \textbf{39} (2014), no.~6, 1128--1157.

\bibitem{HadamaHong2025}
S.~Hadama and Y.~Hong, \emph{Global well-posedness of the nonlinear Hartree equation for infinitely many particles with singular interaction}, J.\ Funct.\ Anal.\ \textbf{289} (2025), no.~9, 111102.

\bibitem{HadamaRout2026}
S.~Hadama and A.~Rout, \emph{Applications of renormalisation to orthonormal Strichartz estimates and the NLS system on the circle}, arXiv:2604.00252 [math.AP], 2026.

\bibitem{HoffmannOstenhof1977}
M.~Hoffmann-Ostenhof and T.~Hoffmann-Ostenhof, \emph{``Schr\"odinger inequalities and asymptotic behaviour of the electron density of atoms and molecules''}, Phys.\ Rev.\ A \textbf{16} (1977), no.~5, 1782--1785.

\bibitem{MI}
T.~B.~Benjamin and J.~E.~Feir, \emph{The disintegration of wave trains on deep water. Part 1. Theory}, J.\ Fluid Mech.\ \textbf{27} (1967), no.~3, 417--430.




\bibitem{LewinSabin2015}
M.~Lewin and J.~Sabin, \emph{The Hartree equation for infinitely many particles. I.\ Well-posedness theory}, Comm.\ Math.\ Phys.\ \textbf{334} (2015), no.~1, 117--170.

\bibitem{LewinSabin2014}
M.~Lewin and J.~Sabin,
\emph{The Hartree equation for infinitely many particles.
II.\ Dispersion and scattering in 2D},
Anal.\ PDE \textbf{7} (2014), no.~6, 1339--1363.

\bibitem{Ochi}
M.~K.~Ochi, \emph{Ocean Waves: The Stochastic Approach}, Cambridge Ocean Technology Series, vol.~6, Cambridge University Press, 1998.

\bibitem{RibalBabaninYoungToffoliStiassnie2013}
A.~Ribal, A.~V.~Babanin, I.~Young, A.~Toffoli, and M.~Stiassnie, \emph{Recurrent solutions of the Alber equation initialized by Joint North Sea Wave Project spectra}, J.\ Fluid Mech.\ \textbf{719} (2013), 314--344.

\bibitem{Smith2024}
M.~Smith, \emph{Phase mixing for the Hartree equation and Landau damping in the semiclassical limit}, arXiv:2412.14842, 2024.


\bibitem{Simon2005}
B.~Simon, \emph{Trace Ideals and Their Applications}, 2nd ed., Math.\ Surveys Monogr., vol.~120, Amer.\ Math.\ Soc., Providence, RI, 2005.

\bibitem{StiassnieRegevAgnon2008}
M.~Stiassnie, A.~Regev, and Y.~Agnon, \emph{Recurrent solutions of Alber's equation for random water-wave fields}, J.\ Fluid Mech.\ \textbf{598} (2008), 245--266.

\bibitem{SulemSulem1999}
C.~Sulem and P.-L.~Sulem, \emph{The Nonlinear Schr\"odinger Equation: Self-Focusing and Wave Collapse}, Applied Mathematical Sciences, vol.~139, Springer, 1999.

\bibitem{vonNeumann1927}
J.~von Neumann, \emph{Wahrscheinlichkeitstheoretischer Aufbau der Quantenmechanik}, Nachr.\ Ges.\ Wiss.\ G\"ottingen Math.-Phys.\ Kl.\ (1927), 245--272.


\bibitem{Zagatti1992}
S.~Zagatti, \emph{The Cauchy problem for Hartree--Fock time-dependent equations}, Ann.\ Inst.\ H.\ Poincar\'e Phys.\ Th\'eor.\ \textbf{56} (1992), no.~4, 357--374.

\bibitem{ZakharovOstrovsky2009}
V.~E.~Zakharov and L.~A.~Ostrovsky, \emph{Modulation instability: the beginning}, Phys.\ D \textbf{238} (2009), no.~5, 540--548.
\end{thebibliography}
\end{document}